\documentclass[10pt]{article}%
\usepackage{amsfonts}
\usepackage{amsmath,amssymb}
\usepackage{amsmath}
\usepackage{amssymb}
\usepackage{graphicx}%
\providecommand{\U}[1]{\protect\rule{.1in}{.1in}}
\newtheorem{theorem}{Theorem}[section]

\newtheorem{definition}[theorem]{Definition}
\newtheorem{example}[theorem]{Example}

\newtheorem{proposition}[theorem]{Proposition}
\newtheorem{remark}[theorem]{Remark}
\newenvironment{proof}[1][Proof]{\noindent \textbf{#1.} }{\  $\Box$}
\numberwithin{equation}{section}
\begin{document}

\title{Comparison Theorem, Feynman-Kac Formula and Girsanov Transformation for BSDEs
Driven by $G$-Brownian Motion}
\author{Mingshang Hu \thanks{School of Mathematics, Shandong University,
humingshang@sdu.edu.cn. Research supported by the National Natural Science
Foundation of China (11201262)}
\and Shaolin Ji\thanks{Qilu Institute of Finance, Shandong University,
jsl@sdu.edu.cn }
\and Shige Peng\thanks{School of Mathematics and Qilu Institute of Finance,
Shandong University, peng@sdu.edu.cn, Hu, Ji, and Peng's research was
partially supported by NSF of China No. 10921101; and by the 111 Project No.
B12023}
\and Yongsheng Song\thanks{Academy of Mathematics and Systems Science, CAS,
Beijing, China, yssong@amss.ac.cn. Research supported by by NCMIS; Youth Grant
of National Science Foundation (No. 11101406); Key Lab of Random Complex
Structures and Data Science, CAS (No. 2008DP173182).} }
\maketitle
\date{}

\begin{abstract}
In this paper, we study comparison theorem, nonlinear Feynman-Kac formula and
Girsanov transformation of the following BSDE driven by a $G$-Brownian
motion.
\begin{align*}
Y_{t} &  =\xi+\int_{t}^{T}f(s,Y_{s},Z_{s})ds+\int_{t}^{T}g(s,Y_{s}%
,Z_{s})d\langle B\rangle_{s}\\
&  -\int_{t}^{T}Z_{s}dB_{s}-(K_{T}-K_{t}),
\end{align*}
where $K$ is a decreasing $G$-martingale.

\end{abstract}


\textbf{MSC-classification}: 60H10, 60H30

\section{Introduction}

Recently, Peng systemically established a time-consistent fully nonlinear
expectation theory (see \cite{Peng2004}, \cite{Peng2005} and \cite{P09}).

As a typical and important case, Peng (2006) introduced the
$G$-expectation theory(see \cite{P10} and the references therein).
In the  $G$-expectation framework ($G$-framework for short), the
notion of  $G$-Brownian motion and the corresponding stochastic
calculus of It\^{o}'s type were established.

The solution of a BSDE driven by $G$-Brownian motion consists of a
triple of processes $(Y,Z,K)$, satisfying
\begin{align}
Y_{t}  &  =\xi+\int_{t}^{T}f(s,Y_{s},Z_{s})ds+\int_{t}^{T}g(s,Y_{s}%
,Z_{s})d\langle B\rangle_{s}\label{e1}\\
&  -\int_{t}^{T}Z_{s}dB_{s}-(K_{T}-K_{t}).\nonumber
\end{align}
The existence and uniqueness of the solution $(Y,Z,K)$ for
(\ref{e1}) is proved in \cite{HJPS}. In this paper, we further
consider the related topics associated with this kind of $G$-BSDEs.

We first study the comparison theorem which is one of the most
important properties of BSDEs. In order to prove this theorem, the
expilcit solutions of linear $G$-BSDEs are obtained. In order to do
this it seems that we have to define the dual forward equations in
an extended $G$-expectation space if the linear $G$-BSDEs include
the $ds$ term. The Gronwall inequality is derived as a by-product
which is interesting by itself.

Then we explore the link between $G$-BSDEs and partial differential
equations (PDE for short). It is well known that under a strong
elliptic assumption, Peng \cite{Peng1991} established a
probabilistic interpretation of a system of quasi-linear PDEs via
classical BSDEs. Then Peng \cite{Peng1992} and Pardoux \& Peng
\cite{PP92} obtained this interpretation for possibly degenerate
situation. This interpretation establishes a one to one
correspondence between the solution of a PDE and the corresponding
classical BSDE, i.e. the so-called nonlinear Feynman-Kac formula.
Peng gave the nonlinear Feynman-Kac Formula for a special type of
$G$-BSDEs in \cite{P10}. In this paper, we consider the
following type of $G$-FBSDEs:%
\[
dX_{s}^{t,\xi}=b(s,X_{s}^{t,\xi})ds+h_{ij}(s,X_{s}^{t,\xi})d\langle
B^{i},B^{j}\rangle_{s}+\sigma_{j}(s,X_{s}^{t,\xi})dB_{s}^{j},\ X_{t}^{t,\xi
}=\xi,
\]%
\begin{align*}
Y_{s}^{t,\xi}  &  =\Phi(X_{T}^{t,\xi})+\int_{s}^{T}f(r,X_{r}^{t,\xi}%
,Y_{r}^{t,\xi},Z_{r}^{t,\xi})dr+\int_{s}^{T}g_{ij}(r,X_{r}^{t,\xi}%
,Y_{r}^{t,\xi},Z_{r}^{t,\xi})d\langle B^{i},B^{j}\rangle_{r}\\
&  -\int_{s}^{T}Z_{r}^{t,\xi}dB_{r}-(K_{T}^{t,\xi}-K_{s}^{t,\xi}),
\end{align*}
Set $u(t,x):=Y_{t}^{t,x}$. We prove that $u(t,x)$ is the unique viscosity
solution of the following PDE:%
\[
\left\{
\begin{array}
[c]{l}%
\partial_{t}u+F(D_{x}^{2}u,D_{x}u,u,x,t)=0,\\
u(T,x)=\Phi(x),
\end{array}
\right.
\]
where%
\begin{align*}
F(D_{x}^{2}u,D_{x}u,u,x,t)=  &  G(H(D_{x}^{2}u,D_{x}u,u,x,t))+\langle
b(t,x),D_{x}u\rangle\\
&  +f(t,x,u,\langle\sigma_{1}(t,x),D_{x}u\rangle,\ldots,\langle\sigma
_{d}(t,x),D_{x}u\rangle),
\end{align*}

\begin{align*}
H_{ij}(D_{x}^{2}u,D_{x}u,u,x,t)=  &  \langle D_{x}^{2}u\sigma_{i}%
(t,x),\sigma_{j}(t,x)\rangle+2\langle D_{x}u,h_{ij}(t,x)\rangle\\
&
+2g_{ij}(t,x,u,\langle\sigma_{1}(t,x),D_{x}u\rangle,\ldots,\langle
\sigma_{d}(t,x),D_{x}u\rangle).
\end{align*}

Finally, we study the Girsanov transformation. Different from \cite{Os} and
\cite{XSZ}, we discuss the Girsanov transformation of the following form:%
\[
\bar{B}_{t}:=B_{t}-\int_{0}^{t}b_{s}ds-\int_{0}^{t}d_{s}^{ij}d\langle
B^{i},B^{j}\rangle_{s}.
\]
We give a direct and simple method to prove that $\bar{B}_{t}$ is a
$G$-Brownian motion under a consistent sublinear expectation.

The paper is organized as follows. In section 2, we present some
preliminaries for stochastic calculus under $G$-framework. The
explicit solutions of linear $G$-BSDEs and the comparison theorem
are established in section 3. In section 4, we obtain the nonlinear
Feynman-Kac formula for a fully nonlinear PDE. We prove the Girsanov
transformation for $G$-Brownian motion in section 5.

\section{Preliminaries}

We review some basic notions and results of $G$-expectation, the related
spaces of random variables and the backward stochastic differential equations
driven by a $G$-Browninan motion. The readers may refer to \cite{HJPS},
\cite{P07a}, \cite{P07b}, \cite{P08a}, \cite{P08b}, \cite{P10} for more details.

\begin{definition}
\label{def2.1} Let $\Omega$ be a given set and let $\mathcal{H}$ be a vector
lattice of real valued functions defined on $\Omega$, namely $c\in\mathcal{H}$
for each constant $c$ and $|X|\in\mathcal{H}$ if $X\in\mathcal{H}$.
$\mathcal{H}$ is considered as the space of random variables. A sublinear
expectation $\mathbb{\hat{E}}$ on $\mathcal{H}$ is a functional $\mathbb{\hat
{E}}:\mathcal{H}\rightarrow\mathbb{R}$ satisfying the following properties:
for all $X,Y\in\mathcal{H}$, we have

\begin{description}
\item[(a)] Monotonicity: If $X\geq Y$ then $\mathbb{\hat{E}}[X]\geq
\mathbb{\hat{E}}[Y]$;

\item[(b)] Constant preservation: $\mathbb{\hat{E}}[c]=c$;

\item[(c)] Sub-additivity: $\mathbb{\hat{E}}[X+Y]\leq\mathbb{\hat{E}%
}[X]+\mathbb{\hat{E}}[Y]$;

\item[(d)] Positive homogeneity: $\mathbb{\hat{E}}[\lambda X]=\lambda
\mathbb{\hat{E}}[X]$ for each $\lambda\geq0$.
\end{description}

$(\Omega,\mathcal{H},\mathbb{\hat{E}})$ is called a sublinear expectation space.
\end{definition}

\begin{definition}
\label{def2.2} Let $X_{1}$ and $X_{2}$ be two $n$-dimensional random vectors
defined respectively in sublinear expectation spaces $(\Omega_{1}%
,\mathcal{H}_{1},\mathbb{\hat{E}}_{1})$ and $(\Omega_{2},\mathcal{H}%
_{2},\mathbb{\hat{E}}_{2})$. They are called identically distributed, denoted
by $X_{1}\overset{d}{=}X_{2}$, if $\mathbb{\hat{E}}_{1}[\varphi(X_{1}%
)]=\mathbb{\hat{E}}_{2}[\varphi(X_{2})]$, for all$\ \varphi\in C_{l.Lip}%
(\mathbb{R}^{n})$, where $C_{l.Lip}(\mathbb{R}^{n})$ is the space of real
continuous functions defined on $\mathbb{R}^{n}$ such that
\[
|\varphi(x)-\varphi(y)|\leq C(1+|x|^{k}+|y|^{k})|x-y|\ \text{\ for
all}\ x,y\in\mathbb{R}^{n},
\]
where $k$ and $C$ depend only on $\varphi$.
\end{definition}

\begin{definition}
\label{def2.3} In a sublinear expectation space $(\Omega,\mathcal{H}%
,\mathbb{\hat{E}})$, a random vector $Y=(Y_{1},\cdot\cdot\cdot,Y_{n})$,
$Y_{i}\in\mathcal{H}$, is said to be independent of another random vector
$X=(X_{1},\cdot\cdot\cdot,X_{m})$, $X_{i}\in\mathcal{H}$ under $\mathbb{\hat
{E}}[\cdot]$, denoted by $Y\bot X$, if for every test function $\varphi\in
C_{l.Lip}(\mathbb{R}^{m}\times\mathbb{R}^{n})$ we have $\mathbb{\hat{E}%
}[\varphi(X,Y)]=\mathbb{\hat{E}}[\mathbb{\hat{E}}[\varphi(x,Y)]_{x=X}]$.
\end{definition}

\begin{definition}
\label{def2.4} ($G$-normal distribution) A $d$-dimensional random vector
$X=(X_{1},\cdot\cdot\cdot,X_{d})$ in a sublinear expectation space
$(\Omega,\mathcal{H},\mathbb{\hat{E}})$ is called $G$-normally distributed if
for each $a,b\geq0$ we have
\[
aX+b\bar{X}\overset{d}{=}\sqrt{a^{2}+b^{2}}X,
\]
where $\bar{X}$ is an independent copy of $X$, i.e., $\bar{X}\overset{d}{=}X$
and $\bar{X}\bot X$. Here the letter $G$ denotes the function
\[
G(A):=\frac{1}{2}\mathbb{\hat{E}}[\langle AX,X\rangle]:\mathbb{S}%
_{d}\rightarrow\mathbb{R},
\]
where $\mathbb{S}_{d}$ denotes the collection of $d\times d$ symmetric matrices.
\end{definition}

Peng \cite{P08b} showed that $X=(X_{1},\cdot\cdot\cdot,X_{d})$ is $G$-normally
distributed if and only if for each $\varphi\in C_{l.Lip}(\mathbb{R}^{d})$,
$u(t,x):=\mathbb{\hat{E}}[\varphi(x+\sqrt{t}X)]$, $(t,x)\in\lbrack
0,\infty)\times\mathbb{R}^{d}$, is the solution of the following $G$-heat
equation:%
\[
\partial_{t}u-G(D_{x}^{2}u)=0,\ u(0,x)=\varphi(x).
\]

The function $G(\cdot):\mathbb{S}_{d}\rightarrow\mathbb{R}$ is a monotonic,
sublinear mapping on $\mathbb{S}_{d}$ and $G(A)=\frac{1}{2}\mathbb{\hat{E}%
}[\langle AX,X\rangle]\leq\frac{1}{2}|A|\mathbb{\hat{E}}[|X|^{2}]=:\frac{1}%
{2}|A|\bar{\sigma}^{2}$ implies that there exists a bounded, convex and closed
subset $\Gamma\subset\mathbb{S}_{d}^{+}$ such that
\[
G(A)=\frac{1}{2}\sup_{\gamma\in\Gamma}\mathrm{tr}[\gamma A],
\]
where $\mathbb{S}_{d}^{+}$ denotes the collection of nonnegative elements in
$\mathbb{S}_{d}$.

In this paper, we only consider non-degenerate $G$-normal distribution, i.e.,
there exists some $\underline{\sigma}^{2}>0$ such that $G(A)-G(B)\geq
\underline{\sigma}^{2}\mathrm{tr}[A-B]$ for any $A\geq B$.

\begin{definition}
\label{def2.5} i) Let $\Omega_{T}=C_{0}([0,T];\mathbb{R}^{d})$, the space of
real valued continuous functions on $[0,T]$ with $\omega_{0}=0$, be endowed
with the supremum norm and let $B_{t}(\omega)=\omega_{t}$ be the canonical
process. Set
\[
\mathcal{H}_{T}^{0}:=\{ \varphi(B_{t_{1}},...,B_{t_{n}}):n\geq1,t_{1}%
,...,t_{n}\in\lbrack0,T],\varphi\in C_{l.Lip}(\mathbb{R}^{d\times n})\}.
\]
Let $G:\mathbb{S}_{d}\rightarrow\mathbb{R}$ be a given monotonic and sublinear
function. $G$-expectation is a sublinear expectation defined by
\[
\mathbb{\hat{E}}[X]=\mathbb{\tilde{E}}[\varphi(\sqrt{t_{1}-t_{0}}\xi_{1}%
,\cdot\cdot\cdot,\sqrt{t_{m}-t_{m-1}}\xi_{m})],
\]
for all $X=\varphi(B_{t_{1}}-B_{t_{0}},B_{t_{2}}-B_{t_{1}},\cdot\cdot
\cdot,B_{t_{m}}-B_{t_{m-1}})$, where $\xi_{1},\cdot\cdot\cdot,\xi_{n}$ are
identically distributed $d$-dimensional $G$-normally distributed random
vectors in a sublinear expectation space $(\tilde{\Omega},\tilde{\mathcal{H}%
},\mathbb{\tilde{E}})$ such that $\xi_{i+1}$ is independent of $(\xi_{1}%
,\cdot\cdot\cdot,\xi_{i})$ for every $i=1,\cdot\cdot\cdot,m-1$. The
corresponding canonical process $B_{t}=(B_{t}^{i})_{i=1}^{d}$ is called a
$G$-Brownian motion.

ii) Let us define the conditional $G$-expectation $\mathbb{\hat{E}}_{t}$ of
$\xi\in\mathcal{H}_{T}^{0}$ knowing $\mathcal{H}_{t}^{0}$, for $t\in
\lbrack0,T]$. Without loss of generality we can assume that $\xi$ has the
representation $\xi=\varphi(B_{t_{1}}-B_{t_{0}},B_{t_{2}}-B_{t_{1}},\cdot
\cdot\cdot,B_{t_{m}}-B_{t_{m-1}})$ with $t=t_{i}$, for some $1\leq i\leq m$,
and we put
\[
\mathbb{\hat{E}}_{t_{i}}[\varphi(B_{t_{1}}-B_{t_{0}},B_{t_{2}}-B_{t_{1}}%
,\cdot\cdot\cdot,B_{t_{m}}-B_{t_{m-1}})]
\]%
\[
=\tilde{\varphi}(B_{t_{1}}-B_{t_{0}},B_{t_{2}}-B_{t_{1}},\cdot\cdot
\cdot,B_{t_{i}}-B_{t_{i-1}}),
\]
where
\[
\tilde{\varphi}(x_{1},\cdot\cdot\cdot,x_{i})=\mathbb{\hat{E}}[\varphi
(x_{1},\cdot\cdot\cdot,x_{i},B_{t_{i+1}}-B_{t_{i}},\cdot\cdot\cdot,B_{t_{m}%
}-B_{t_{m-1}})].
\]

\end{definition}

Define $\Vert\xi\Vert_{p,G}=(\mathbb{\hat{E}}[|\xi|^{p}])^{1/p}$ for $\xi
\in\mathcal{H}_{T}^{0}$ and $p\geq1$. Then \textmd{for all}$\ t\in\lbrack
0,T]$, $\mathbb{\hat{E}}_{t}[\cdot]$ is a continuous mapping on $\mathcal{H}%
_{T}^{0}$ w.r.t. the norm $\Vert\cdot\Vert_{1,G}$. Therefore it can be
extended continuously to the completion $L_{G}^{1}(\Omega_{T})$ of
$\mathcal{H}_{T}^{0}$ under the norm $\Vert\cdot\Vert_{1,G}$.

Let $L_{ip}(\Omega_{T}):=\{ \varphi(B_{t_{1}},...,B_{t_{n}}):n\geq
1,t_{1},...,t_{n}\in\lbrack0,T],\varphi\in C_{b.Lip}(\mathbb{R}^{d\times
n})\},$ where $C_{b.Lip}(\mathbb{R}^{d\times n})$ denotes the set of bounded
Lipschitz functions on $\mathbb{R}^{d\times n}$. Denis et al. \cite{DHP11}
proved that the completions of $C_{b}(\Omega_{T})$ (the set of bounded
continuous function on $\Omega_{T}$), $\mathcal{H}_{T}^{0}$ and $L_{ip}%
(\Omega_{T})$ under $\Vert\cdot\Vert_{p,G}$ are the same and we denote them by
$L_{G}^{p}(\Omega_{T})$.

For each fixed $\mathbf{a}\in\mathbb{R}^{d}$, $B_{t}^{\mathbf{a}}%
=\langle\mathbf{a},B_{t}\rangle$ is a $1$-dimensional $G_{\mathbf{a}}%
$-Brownian motion, where $G_{\mathbf{a}}(\alpha)=\frac{1}{2}(\sigma
_{\mathbf{aa}^{T}}^{2}\alpha^{+}-\sigma_{-\mathbf{aa}^{T}}^{2}\alpha^{-})$,
$\sigma_{\mathbf{aa}^{T}}^{2}=2G(\mathbf{aa}^{T})$, $\sigma_{-\mathbf{aa}^{T}%
}^{2}=-2G(-\mathbf{aa}^{T})$. Let $\pi_{t}^{N}=\{t_{0}^{N},\cdots,t_{N}^{N}%
\}$, $N=1,2,\cdots$, be a sequence of partitions of $[0,t]$ such that $\mu
(\pi_{t}^{N})=\max\{|t_{i+1}^{N}-t_{i}^{N}|:i=0,\cdots,N-1\} \rightarrow0$,
the quadratic variation process of $B^{\mathbf{a}}$ is defined by%
\[
\langle B^{\mathbf{a}}\rangle_{t}=\lim_{\mu(\pi_{t}^{N})\rightarrow0}%
\sum_{j=0}^{N-1}(B_{t_{j+1}^{N}}^{\mathbf{a}}-B_{t_{j}^{N}}^{\mathbf{a}}%
)^{2}.
\]
For each fixed $\mathbf{a}$, $\mathbf{\bar{a}}\in\mathbb{R}^{d}$, the mutual
variation process of $B^{\mathbf{a}}$ and $B^{\mathbf{\bar{a}}}$ is defined by%
\[
\langle B^{\mathbf{a}},B^{\mathbf{\bar{a}}}\rangle_{t}=\frac{1}{4}[\langle
B^{\mathbf{a}+\mathbf{\bar{a}}}\rangle_{t}-\langle B^{\mathbf{a}%
-\mathbf{\bar{a}}}\rangle_{t}].
\]

\begin{definition}
\label{def2.6} Let $M_{G}^{0}(0,T)$ be the collection of processes in the
following form: for a given partition $\{t_{0},\cdot\cdot\cdot,t_{N}\}=\pi
_{T}$ of $[0,T]$,
\[
\eta_{t}(\omega)=\sum_{j=0}^{N-1}\xi_{j}(\omega)I_{[t_{j},t_{j+1})}(t),
\]
where $\xi_{i}\in L_{ip}(\Omega_{t_{i}})$, $i=0,1,2,\cdot\cdot\cdot,N-1$. For
$p\geq1$ and $\eta\in M_{G}^{0}(0,T)$, let $\Vert\eta\Vert_{H_{G}^{p}}=\{
\mathbb{\hat{E}}[(\int_{0}^{T}|\eta_{s}|^{2}ds)^{p/2}]\}^{1/p}$, $\Vert
\eta\Vert_{M_{G}^{p}}=\{ \mathbb{\hat{E}}[\int_{0}^{T}|\eta_{s}|^{p}%
ds]\}^{1/p}$ and denote by $H_{G}^{p}(0,T)$, $M_{G}^{p}(0,T)$ the completions
of $M_{G}^{0}(0,T)$ under the norms $\Vert\cdot\Vert_{H_{G}^{p}}$, $\Vert
\cdot\Vert_{M_{G}^{p}}$ respectively.
\end{definition}

\begin{theorem}
\label{the2.7} (\cite{DHP11,HP09}) There exists a weakly compact set
$\mathcal{P}\subset\mathcal{M}_{1}(\Omega_{T})$, the set of probability
measures on $(\Omega_{T},\mathcal{B}(\Omega_{T}))$, such that
\[
\mathbb{\hat{E}}[\xi]=\sup_{P\in\mathcal{P}}E_{P}[\xi]\ \ \text{for
\ all}\ \xi\in\mathcal{H}_{T}^{0}.
\]
$\mathcal{P}$ is called a set that represents $\mathbb{\hat{E}}$.
\end{theorem}

Let $\mathcal{P}$ be a weakly compact set that represents $\mathbb{\hat{E}}$.
For this $\mathcal{P}$, we define capacity%
\[
c(A):=\sup_{P\in\mathcal{P}}P(A),\ A\in\mathcal{B}(\Omega_{T}).
\]
A set $A\subset\Omega_{T}$ is polar if $c(A)=0$. A property holds
\textquotedblleft quasi-surely\textquotedblright\ (q.s. for short) if it holds
outside a polar set. In the following, we do not distinguish two random
variables $X$ and $Y$ if $X=Y$ q.s.. We set%
\[
\mathbb{L}^{p}(\Omega_{t}):=\{X\in\mathcal{B}(\Omega_{t}):\sup_{P\in
\mathcal{P}}E_{P}[|X|^{p}]<\infty\} \ \text{for}\ p\geq1.
\]
It is important to note that $L_{G}^{p}(\Omega_{t})\subset\mathbb{L}%
^{p}(\Omega_{t})$. We extend $G$-expectation $\mathbb{\hat{E}}$ to
$\mathbb{L}^{p}(\Omega_{t})$ and still denote it by $\mathbb{\hat{E}}$, for
each $X\in$ $\mathbb{L}^{1}(\Omega_{T})$, we set%
\[
\mathbb{\hat{E}}[X]=\sup_{P\in\mathcal{P}}E_{P}[X].
\]
For $p\geq1$, $\mathbb{L}^{p}(\Omega_{t})$ is a Banach space under the norm
$(\mathbb{\hat{E}}[|\cdot|^{p}])^{1/p}$.

Set%
\[
\mathbb{L}_{G}^{0,p,t}(\Omega_{T}):=\{ \xi=\sum_{i=1}^{n}\eta_{i}I_{A_{i}%
}:A_{i}\in\mathcal{B}(\Omega_{t}),\eta_{i}\in L_{G}^{p}(\Omega),n\in
\mathbb{N}\},
\]
we define the corresponding conditional $G$-expectation, still denoted by
$\mathbb{\hat{E}}_{s}[\cdot]$, by setting%
\[
\mathbb{\hat{E}}_{s}[\sum_{i=1}^{n}\eta_{i}I_{A_{i}}]:=\sum_{i=1}%
^{n}\mathbb{\hat{E}}_{s}[\eta_{i}]I_{A_{i}}\ \text{\ for}\ s\geq t.
\]

\begin{proposition}
\label{proA.8} (\cite{HJPS}) For each $\xi,\eta\in\mathbb{L}_{G}%
^{0,1,t}(\Omega_{T})$, we have

\begin{description}
\item[(i)] Monotonicity: If $\xi\leq\eta$, then $\mathbb{\hat{E}}_{s}[\xi
]\leq\mathbb{\hat{E}}_{s}[\eta]$ for any $s\geq t$;

\item[(ii)] Constant preserving: If $\xi\in\mathbb{L}_{G}^{0,1,t}(\Omega_{t})
$, then $\mathbb{\hat{E}}_{t}[\xi]=\xi$;

\item[(iii)] Sub-additivity: $\mathbb{\hat{E}}_{s}[\xi+\eta]\leq
\mathbb{\hat{E}}_{s}[\xi]+\mathbb{\hat{E}}_{s}[\eta]$ for any $s\geq t$;

\item[(iv)] Positive homogeneity: If $\xi\in\mathbb{L}_{G}^{0,\infty,t}%
(\Omega_{t})$ and $\xi\geq0$, then $\mathbb{\hat{E}}_{t}[\xi\eta
]=\xi\mathbb{\hat{E}}_{t}[\eta]$;

\item[(v)] Consistency: For $t\leq s\leq r$, we have $\mathbb{\hat{E}}%
_{s}[\mathbb{\hat{E}}_{r}[\xi]]=\mathbb{\hat{E}}_{s}[\xi]$.

\item[(vi)] $\mathbb{\hat{E}}[\mathbb{\hat{E}}_{t}[\xi]]=\mathbb{\hat{E}}%
[\xi]$.
\end{description}
\end{proposition}

Let $\mathbb{L}_{G}^{p,t}(\Omega_{T})$ be the completion of $\mathbb{L}%
_{G}^{0,p,t}(\Omega_{T})$ under the norm $(\mathbb{\hat{E}}[|\cdot
|^{p}])^{1/p}$. Clearly, the conditional $G$-expectation can be extended
continuously to $\mathbb{L}_{G}^{1,t}(\Omega_{T})$.

Set%
\[
\mathbb{M}^{p,0}(0,T):=\{ \eta_{t}=\sum_{i=0}^{N-1}\xi_{t_{i}}I_{[t_{i}%
,t_{i+1})}(t):0=t_{0}<\cdots<t_{N}=T,\xi_{t_{i}}\in\mathbb{L}^{p}%
(\Omega_{t_{i}})\}.
\]
For $p\geq1$, we denote by $\mathbb{M}^{p}(0,T)$, $\mathbb{H}^{p}(0,T)$,
$\mathbb{S}^{p}(0,T)$ the completion of $\mathbb{M}^{p,0}(0,T)$ under the norm
$||\eta||_{\mathbb{M}^{p}}:=(\mathbb{\hat{E}}[\int_{0}^{T}|\eta_{t}%
|^{p}dt])^{1/p}$, $||\eta||_{\mathbb{H}^{p}}:=\{ \mathbb{\hat{E}}[(\int%
_{0}^{T}|\eta_{t}|^{2}dt)^{p/2}]\}^{1/p}$, $||\eta||_{\mathbb{S}^{p}%
}:=(\mathbb{\hat{E}}[\sup_{t\in\lbrack0,T]}|\eta_{t}|^{p}])^{1/p}$
respectively. Following Li and Peng \cite{L-P}, for each $\eta\in
\mathbb{H}^{p}(0,T)$ with $p\geq1$, we can define It\^{o}'s integral $\int%
_{0}^{T}\eta_{s}dB_{s}$. Moreover, by Proposition 2.10 in \cite{L-P} and
classical Burkholder-Davis-Gundy Inequality, the following properties hold.

\begin{proposition}
\label{proA.5} For each $\eta,\theta\in\mathbb{H}^{\alpha}(0,T)$ with
$\alpha\geq1$ and $p>0$, $\xi\in\mathbb{L}^{\infty}(\Omega_{t})$, we have%
\[
\mathbb{\hat{E}}[\int_{0}^{T}\eta_{s}dB_{s}]=0,
\]%
\[
\ \underline{\sigma}^{p}c_{p}\mathbb{\hat{E}}[(\int_{0}^{T}|\eta_{s}%
|^{2}ds)^{p/2}]\leq\mathbb{\hat{E}}[\sup_{t\in\lbrack0,T]}|\int_{0}^{t}%
\eta_{s}dB_{s}|^{p}]\leq\bar{\sigma}^{p}C_{p}\mathbb{\hat{E}}[(\int_{0}%
^{T}|\eta_{s}|^{2}ds)^{p/2}],
\]%
\[
\int_{t}^{T}(\xi\eta_{s}+\theta_{s})dB_{s}=\xi\int_{t}^{T}\eta_{s}dB_{s}%
+\int_{t}^{T}\theta_{s}dB_{s},
\]
where $0<c_{p}<C_{p}<\infty$ are constants.
\end{proposition}

\begin{remark}
\label{nnrem1} If $\eta\in H_{G}^{\alpha}(0,T)$ with $\alpha\geq1$ and
$p\in(0,\alpha]$, then we can get $\sup_{u\in\lbrack t,T]}|\int_{t}^{u}%
\eta_{s}dB_{s}|^{p}\in L_{G}^{1}(\Omega_{T})$ and
\[
\underline{\sigma}^{p}c_{p}\mathbb{\hat{E}}_{t}[(\int_{t}^{T}|\eta_{s}%
|^{2}ds)^{p/2}]\leq\mathbb{\hat{E}}_{t}[\sup_{u\in\lbrack t,T]}|\int_{t}%
^{u}\eta_{s}dB_{s}|^{p}]\leq\bar{\sigma}^{p}C_{p}\mathbb{\hat{E}}_{t}%
[(\int_{t}^{T}|\eta_{s}|^{2}ds)^{p/2}].
\]

\end{remark}

\begin{definition}
\label{def2.9} A process $\{M_{t}\}$ with values in $L_{G}^{1}(\Omega_{T})$ is
called a $G$-martingale if $\mathbb{\hat{E}}_{s}[M_{t}]=M_{s}$ for any $s\leq
t$.
\end{definition}

Let $S_{G}^{0}(0,T)=\{h(t,B_{t_{1}\wedge t},\cdot\cdot\cdot,B_{t_{n}\wedge
t}):t_{1},\ldots,t_{n}\in\lbrack0,T],h\in C_{b,Lip}(\mathbb{R}^{n+1})\}$. For
$p\geq1$ and $\eta\in S_{G}^{0}(0,T)$, set $\Vert\eta\Vert_{S_{G}^{p}}=\{
\mathbb{\hat{E}}[\sup_{t\in\lbrack0,T]}|\eta_{t}|^{p}]\}^{\frac{1}{p}}$.
Denote by $S_{G}^{p}(0,T)$ the completion of $S_{G}^{0}(0,T)$ under the norm
$\Vert\cdot\Vert_{S_{G}^{p}}$.

We consider the following type of $G$-BSDEs (in this paper we always use Einstein convention):%
\begin{align}
Y_{t}  &  =\xi+\int_{t}^{T}f(s,Y_{s},Z_{s})ds+\int_{t}^{T}g_{ij}(s,Y_{s}%
,Z_{s})d\langle B^{i},B^{j}\rangle_{s}\nonumber\\
&  -\int_{t}^{T}Z_{s}dB_{s}-(K_{T}-K_{t}), \label{e3}%
\end{align} where

\[
f(t,\omega,y,z),g_{ij}(t,\omega,y,z):[0,T]\times\Omega_{T}\times
\mathbb{R}\times\mathbb{R}^{d}\rightarrow\mathbb{R}%
\]
satisfy the following properties:

\begin{description}
\item[(H1)] There exists some $\beta>1$ such that for any $y,z$,
$f(\cdot,\cdot,y,z),g_{ij}(\cdot,\cdot,y,z)\in M_{G}^{\beta}(0,T)$;

\item[(H2)] There exists some $L>0$ such that
\[
|f(t,y,z)-f(t,y^{\prime},z^{\prime})|+\sum_{i,j=1}^{d}|g_{ij}(t,y,z)-g_{ij}%
(t,y^{\prime},z^{\prime})|\leq L(|y-y^{\prime}|+|z-z^{\prime}|).
\]

\end{description}

For simplicity, we denote by $\mathfrak{S}_{G}^{\alpha}(0,T)$ the collection
of processes $(Y,Z,K)$ such that $Y\in S_{G}^{\alpha}(0,T)$, $Z\in
H_{G}^{\alpha}(0,T;\mathbb{R}^{d})$, $K$ is a decreasing $G$-martingale with
$K_{0}=0$ and $K_{T}\in L_{G}^{\alpha}(\Omega_{T})$.

\begin{definition}
\label{def3.1} Let $\xi\in L_{G}^{\beta}(\Omega_{T})$ and $f$ satisfy (H1) and
(H2) for some $\beta>1$. A triplet of processes $(Y,Z,K)$ is called a solution
of equation (\ref{e3}) if for some $1<\alpha\leq\beta$ the following
properties hold:

\begin{description}
\item[(a)] $(Y,Z,K)\in\mathfrak{S}_{G}^{\alpha}(0,T)$;

\item[(b)] $Y_{t}=\xi+\int_{t}^{T}f(s,Y_{s},Z_{s})ds+\int_{t}^{T}%
g_{ij}(s,Y_{s},Z_{s})d\langle B^{i},B^{j}\rangle_{s}-\int_{t}^{T}Z_{s}%
dB_{s}-(K_{T}-K_{t})$.
\end{description}
\end{definition}

\begin{theorem}
\label{the4.1} (\cite{HJPS}) Assume that $\xi\in L_{G}^{\beta}(\Omega_{T})$
and $f$, $g_{ij}$ satisfy (H1) and (H2) for some $\beta>1$. Then equation
(\ref{e3}) has a unique solution $(Y,Z,K)$. Moreover, for any $1<\alpha<\beta$
we have $Y\in S_{G}^{\alpha}(0,T)$, $Z\in H_{G}^{\alpha}(0,T;\mathbb{R}^{d})$
and $K_{T}\in L_{G}^{\alpha}(\Omega_{T})$.
\end{theorem}

We have the following estimates.

\begin{proposition}
\label{pro3.4} (\cite{HJPS}) Let $\xi\in L_{G}^{\beta}(\Omega_{T})$ and $f$,
$g_{ij}$ satisfy (H1) and (H2) for some $\beta>1$. Assume that $(Y,Z,K)\in
\mathfrak{S}_{G}^{\alpha}(0,T)$ for some $1<\alpha<\beta$ is a solution of
equation (\ref{e3}). Then

\begin{description}
\item[(i)] There exists a constant $C_{\alpha}:=C(\alpha,T,G,L)>0$ such that%
\[
|Y_{t}|^{\alpha}\leq C_{\alpha}\mathbb{\hat{E}}_{t}[|\xi|^{\alpha}+\int%
_{t}^{T}|h_{s}^{0}|^{\alpha}ds],
\]%
\[
\mathbb{\hat{E}}[(\int_{0}^{T}|Z_{s}|^{2}ds)^{\frac{\alpha}{2}}]\leq
C_{\alpha}\{ \mathbb{\hat{E}}[\sup_{t\in\lbrack0,T]}|Y_{t}|^{\alpha
}]+(\mathbb{\hat{E}}[\sup_{t\in\lbrack0,T]}|Y_{t}|^{\alpha}])^{\frac{1}{2}%
}(\mathbb{\hat{E}}[(\int_{0}^{T}h_{s}^{0}ds)^{\alpha}])^{\frac{1}{2}}\},
\]%
\[
\mathbb{\hat{E}}[|K_{T}|^{\alpha}]\leq C_{\alpha}\{ \mathbb{\hat{E}}%
[\sup_{t\in\lbrack0,T]}|Y_{t}|^{\alpha}]+\mathbb{\hat{E}}[(\int_{0}^{T}%
h_{s}^{0}ds)^{\alpha}]\},
\]
where $h_{s}^{0}=|f(s,0,0)|+\sum_{i,j=1}^{d}|g_{ij}(s,0,0)|$.

\item[(ii)] For any given $\alpha^{\prime}$ with $\alpha<\alpha^{\prime}%
<\beta$, there exists a constant $C_{\alpha,\alpha^{\prime}}$ depending on
$\alpha$, $\alpha^{\prime}$, $T$, $G$, $L$ such that%
\begin{align*}
\mathbb{\hat{E}}[\sup_{t\in\lbrack0,T]}|Y_{t}|^{\alpha}]  &  \leq
C_{\alpha,\alpha^{\prime}}\{ \mathbb{\hat{E}}[\sup_{t\in\lbrack0,T]}%
\mathbb{\hat{E}}_{t}[|\xi|^{\alpha}]]\\
&  +(\mathbb{\hat{E}}[\sup_{t\in\lbrack0,T]}\mathbb{\hat{E}}_{t}[(\int_{0}%
^{T}h_{s}^{0}ds)^{\alpha^{\prime}}]])^{\frac{\alpha}{\alpha^{\prime}}%
}+\mathbb{\hat{E}}[\sup_{t\in\lbrack0,T]}\mathbb{\hat{E}}_{t}[(\int_{0}%
^{T}h_{s}^{0}ds)^{\alpha^{\prime}}]]\}.
\end{align*}

\end{description}
\end{proposition}

\begin{proposition}
\label{pro3.5} (\cite{HJPS}) Let $\xi^{l}\in L_{G}^{\beta}(\Omega_{T})$ ,
$l=1,2$, and $f^{l}$, $g_{ij}^{l}$ satisfy (H1) and (H2') for some $\beta>1$.
Assume that $(Y^{l},Z^{l},K^{l})\in\mathfrak{S}_{G}^{\alpha}(0,T)$ for some
$1<\alpha<\beta$ are the solutions of equation (\ref{e3}) corresponding to
$\xi^{l}$ $f^{l}$ and $g_{ij}^{l}$. Set $\hat{Y}_{t}=Y_{t}^{1}-Y_{t}^{2}%
,\hat{Z}_{t}=Z_{t}^{1}-Z_{t}^{2}$ and $\hat{K}_{t}=K_{t}^{1}-K_{t}^{2}$. Then

\begin{description}
\item[(i)] There exists a constant $C_{\alpha}:=C(\alpha,T,G,L)>0$ such that%
\[
|\hat{Y}_{t}|^{\alpha}\leq C_{\alpha}\mathbb{\hat{E}}_{t}[|\hat{\xi}|^{\alpha
}+\int_{t}^{T}|\hat{h}_{s}|^{\alpha}ds],
\]
where $\hat{\xi}=\xi^{1}-\xi^{2}$, $\hat{h}_{s}=|f^{1}(s,Y_{s}^{2},Z_{s}%
^{2})-f^{2}(s,Y_{s}^{2},Z_{s}^{2})|+\sum_{i,j=1}^{d}|g_{ij}^{1}(s,Y_{s}%
^{2},Z_{s}^{2})-g_{ij}^{2}(s,Y_{s}^{2},Z_{s}^{2})|$.

\item[(ii)] For any given $\alpha^{\prime}$ with $\alpha<\alpha^{\prime}%
<\beta$, there exists a constant $C_{\alpha,\alpha^{\prime}}$ depending on
$\alpha$, $\alpha^{\prime}$, $T$, $G$, $L$ such that%
\begin{align*}
\mathbb{\hat{E}}[\sup_{t\in\lbrack0,T]}|\hat{Y}_{t}|^{\alpha}]  &  \leq
C_{\alpha,\alpha^{\prime}}\{ \mathbb{\hat{E}}[\sup_{t\in\lbrack0,T]}%
\mathbb{\hat{E}}_{t}[|\hat{\xi}|^{\alpha}]]\\
&  +(\mathbb{\hat{E}}[\sup_{t\in\lbrack0,T]}\mathbb{\hat{E}}_{t}[(\int_{0}%
^{T}\hat{h}_{s}ds)^{\alpha^{\prime}}]])^{\frac{\alpha}{\alpha^{\prime}}%
}+\mathbb{\hat{E}}[\sup_{t\in\lbrack0,T]}\mathbb{\hat{E}}_{t}[(\int_{0}%
^{T}\hat{h}_{s}ds)^{\alpha^{\prime}}]]\}.
\end{align*}

\end{description}
\end{proposition}

\section{Comparison theorem of $G$-BSDEs}

For simplicity, we consider $1$-dimensional $G$-Brownian motion case. The
results still hold for the case $d>1$.

\subsection{Explicit solutions of linear $G$-BSDEs}

Let $(\Omega_{T},L_{G}^{1}(\Omega_{T}),\mathbb{\hat{E}})$ with $\Omega
_{T}=C_{0}([0,T],\mathbb{R})$ be a $G$-expectation space. We consider the
explicit solution of the following linear $G$-BSDE:
\begin{equation}
Y_{t}=\xi+\int_{t}^{T}f_{s}ds+\int_{t}^{T}g_{s}d\langle B\rangle_{s}-\int%
_{t}^{T}Z_{s}dB_{s}-(K_{T}-K_{t}), \label{LBSDE1}%
\end{equation}
where $f_{s}=a_{s}Y_{s}+b_{s}Z_{s}+m_{s}$, $g_{s}=c_{s}Y_{s}+d_{s}Z_{s}+n_{s}$
with $\{a_{s}\}_{s\in\lbrack0,T]}$, $\{b_{s}\}_{s\in\lbrack0,T]}$,
$\{c_{s}\}_{0\leq s\in\lbrack0,T]}$, $\{d_{s}\}_{s\in\lbrack0,T]}$ bounded
processes in $M_{G}^{\beta}(0,T)$ and $\xi\in L_{G}^{\beta}(\Omega_{T})$,
$\{m_{s}\}_{s\in\lbrack0,T]}$, $\{n_{s}\}_{s\in\lbrack0,T]}\in M_{G}^{\beta
}(0,T)$ with $\beta>1$. For this purpose we construct an auxiliary extended
$\tilde{G}$-expectation space $(\tilde{\Omega}_{T},L_{\tilde{G}}^{1}%
(\tilde{\Omega}_{T}),\mathbb{\hat{E}}^{\tilde{G}})$ with $\tilde{\Omega}%
_{T}=C_{0}([0,T],\mathbb{R}^{2})$ and%

\[
\tilde{G}(A)=\frac{1}{2}\sup_{\underline{\sigma}^{2}\leq v\leq\bar{\sigma}%
^{2}}\mathrm{tr}\left[  A\left[
\begin{array}
[c]{cc}%
v & 1\\
1 & v^{-1}%
\end{array}
\right]  \right]  ,\ A\in\mathbb{S}_{2}.
\]
Let $\{(B_{t},\tilde{B}_{t})\}$ be the canonical process in the extended space.

\begin{remark}
\label{rem5.1} It is easy to check that $\langle B,\tilde{B}\rangle_{t}=t$. In
particular, if $\underline{\sigma}^{2}=\bar{\sigma}^{2}$, we can further get
$\tilde{B}_{t}=\bar{\sigma}^{-2}B_{t}$.
\end{remark}

Let $\{X_{t}\}_{t\in\lbrack0,T]}$ be the solution of the following $\tilde{G}%
$-SDE:%
\begin{equation}
X_{t}=1+\int_{0}^{t}a_{s}X_{s}ds+\int_{0}^{t}c_{s}X_{s}d\langle B\rangle
_{s}+\int_{0}^{t}d_{s}X_{s}dB_{s}+\int_{0}^{t}b_{s}X_{s}d\tilde{B}_{s}.
\label{LSDE2}%
\end{equation}
It is easy to verfy that%
\begin{equation}
X_{t}=\exp(\int_{0}^{t}(a_{s}-b_{s}d_{s})ds+\int_{0}^{t}c_{s}d\langle
B\rangle_{s})\mathcal{E}_{t}^{B}\mathcal{E}_{t}^{\tilde{B}}, \label{LSDE3}%
\end{equation}
where $\mathcal{E}_{t}^{B}=\exp(\int_{0}^{t}d_{s}dB_{s}-\frac{1}{2}\int%
_{0}^{t}d_{s}^{2}d\langle B\rangle_{s})$, $\mathcal{E}_{t}^{\tilde{B}}%
=\exp(\int_{0}^{t}b_{s}d\tilde{B}_{s}-\frac{1}{2}\int_{0}^{t}b_{s}^{2}%
d\langle\tilde{B}\rangle_{s})$.

\begin{theorem}
\label{the5.2} In the extended $\tilde{G}$-expectation space, the solution of
the $G$-BSDE (\ref{LBSDE1}) can be represented as%
\begin{equation}
Y_{t}=(X_{t})^{-1}\mathbb{\hat{E}}_{t}^{\tilde{G}}[X_{T}\xi+\int_{t}^{T}%
m_{s}X_{s}ds+\int_{t}^{T}n_{s}X_{s}d\langle B\rangle_{s}], \label{LBSDE4}%
\end{equation}
where $\{X_{t}\}_{t\in\lbrack0,T]}$ is the solution of the $\tilde{G}$-SDE
(\ref{LSDE2}).
\end{theorem}

\begin{proof}
By applying It\^{o}'s formula to $X_{t}Y_{t}$, we get%
\begin{align*}
&  X_{t}Y_{t}+\int_{t}^{T}(X_{s}Z_{s}+d_{s}X_{s}Y_{s})dB_{s}+\int_{t}^{T}%
b_{s}X_{s}Y_{s}d\tilde{B}_{s}+\int_{t}^{T}X_{s}dK_{s}\\
&  =X_{T}\xi+\int_{t}^{T}m_{s}X_{s}ds+\int_{t}^{T}n_{s}X_{s}d\langle
B\rangle_{s}.
\end{align*}
By Lemma 3.4 in \cite{HJPS}, we have $\{ \int_{0}^{t}X_{s}dK_{s}%
\}_{t\in\lbrack0,T]}$ is a $\tilde{G}$-martingale. Thus we get
\[
Y_{t}=(X_{t})^{-1}\mathbb{\hat{E}}_{t}^{\tilde{G}}[X_{T}\xi+\int_{t}^{T}%
m_{s}X_{s}ds+\int_{t}^{T}n_{s}X_{s}d\langle B\rangle_{s}].
\]

\end{proof}

\begin{remark}
\label{rem5.3} If $b_{t}=0$, the solution of the $G$-BSDE (\ref{LBSDE1}) is
\[
Y_{t}=(X_{t})^{-1}\mathbb{\hat{E}}_{t}[X_{T}\xi+\int_{t}^{T}m_{s}X_{s}%
ds+\int_{t}^{T}n_{s}X_{s}d\langle B\rangle_{s}],
\]
where $X_{t}=\exp(\int_{0}^{t}a_{s}ds+\int_{0}^{t}(c_{s}-\frac{1}{2}d_{s}%
^{2})d\langle B\rangle_{s}+\int_{0}^{t}d_{s}dB_{s})$. In this case, we do not
need to construct an auxiliary space. If $b_{t}\neq0$, the form of $X_{t}$
contains $\tilde{B}$, but
\[
Y_{t}=\mathbb{\hat{E}}_{t}^{\tilde{G}}[X_{T}^{t}\xi+\int_{t}^{T}m_{s}X_{s}%
^{t}ds+\int_{t}^{T}n_{s}X_{s}^{t}d\langle B\rangle_{s}]
\]
does not contain $\tilde{B}$, where $X_{s}^{t}=X_{s}/X_{t}$. For simplicity,
we only give an explanation for $\xi=\varphi(B_{T})$, $f_{s}=b_{s}Z_{s}$ with
$b_{s}=\psi(B_{s})$ and $g_{s}=0$ in the $G$-BSDE (\ref{LBSDE1}). In this
case,%
\begin{align*}
Y_{t}  &  =\mathbb{\hat{E}}_{t}^{\tilde{G}}[\varphi(B_{T})\exp(\int_{t}%
^{T}\psi(B_{s})d\tilde{B}_{s}-\frac{1}{2}\int_{t}^{T}|\psi(B_{s})|^{2}%
d\langle\tilde{B}\rangle_{s})]\\
&  =\mathbb{\hat{E}}^{\tilde{G}}[\varphi(x+B_{T}^{t})\exp(\int_{t}^{T}%
\psi(x+B_{s}^{t})d\tilde{B}_{s}-\frac{1}{2}\int_{t}^{T}|\psi(x+B_{s}^{t}%
)|^{2}d\langle\tilde{B}\rangle_{s})]_{x=B_{t}},
\end{align*}
which does not contain $\tilde{B}$, where $B_{s}^{t}=B_{s}-B_{t}$.

Note that $\mathbb{\hat{E}}^{\tilde{G}}[\xi]=\mathbb{\hat{E}}[\xi]$ for each
$\xi\in L_{G}^{1}(\Omega_{T})$, thus this $Y$ in Theorem \ref{the5.2} is the
solution of the $G$-BSDE (\ref{LBSDE1}) in $(\Omega_{T},L_{G}^{1}(\Omega
_{T}),\mathbb{\hat{E}})$. Here $\tilde{B}$ is an auxiliary process and
disappear by taking conditional expectation.
\end{remark}

\begin{remark}
\label{rem5.4} If $b_{s}=0$, $d_{s}=0$, we have the following special type of
$G$-BSDE:%
\begin{equation}
Y_{t}=\mathbb{\hat{E}}_{t}[\xi+\int_{t}^{T}(a_{s}Y_{s}+m_{s})ds+\int_{t}%
^{T}(c_{s}Y_{s}+n_{s})d\langle B\rangle_{s}], \label{LBSDE5}%
\end{equation}
where $\{a_{s}\}_{s\in\lbrack0,T]}$, $\{c_{s}\}_{s\in\lbrack0,T]}$ are bounded
processes in $M_{G}^{1}(0,T)$ and $\xi\in L_{G}^{1}(\Omega)$, $\{m_{s}%
\}_{s\in\lbrack0,T]}$, $\{n_{s}\}_{s\in\lbrack0,T]}\in M_{G}^{1}(0,T)$. By
applying Theorem \ref{the5.2} to $\xi^{N}=(\xi\wedge N)\vee(-N)$, $m_{s}%
^{N}=(m_{s}\wedge N)\vee(-N)$, $n_{s}^{N}=(n_{s}\wedge N)\vee(-N)$ for each
$N>0$, we obtain that the explicit solution of the $G$-BSDE (\ref{LBSDE5}) is
\begin{equation}
Y_{t}=(X_{t})^{-1}\mathbb{\hat{E}}_{t}[X_{T}\xi+\int_{t}^{T}m_{s}X_{s}%
ds+\int_{t}^{T}n_{s}X_{s}d\langle B\rangle_{s}], \label{LBSDE6}%
\end{equation}
where $X_{t}=\exp(\int_{0}^{t}a_{s}ds+\int_{0}^{t}c_{s}d\langle B\rangle_{s})$.
\end{remark}

In the following, we explain why we have to extend the space. For simplicity,
we only consider%
\[
Y_{t}=\xi+\int_{t}^{T}Z_{s}ds-\int_{t}^{T}Z_{s}dB_{s}-(K_{T}-K_{t}).
\]
In order to get the explicit solution of the above $G$-BSDE, we try to find a
positive process $X$ (not depending on $Y, Z, K$) such that $XY$ is a
$G$-martingale. Applying It\^{o}'s formula to $XY$, we have
\[
d(X_{t}Y_{t})=X_{t}Z_{t}dB_{t}+X_{t}dK_{t}-X_{t}Z_{t}dt+Z_{t}d\langle X,
B\rangle_{t}+Y_{t}dX_{t}.
\]
So as to guarantee that $XY$ is a $G$-martingale, $-X_{t}Z_{t}dt+Z_{t}d\langle
X, B\rangle_{t}+Y_{t}dX_{t}$ should be a symmetric $G$-martingale, which
implies that $X$ is a symmetric $G$-martingale and%
\begin{equation}
\label{exp3}X_{t}dt=d\langle X, B\rangle_{t}.
\end{equation}

By the representation theorem of symmetric $G$-martingales, we assume
$X_{t}=X_{0}+\int_{0}^{t}h_{s}dB_{s}$ for some $h\in M^{2}_{G}(0,T)$. Then
equation (\ref{exp3}) implies that
\[
X_{t}dt=h_{t}d\langle B\rangle_{t}.
\]
By Corollary 3.5 in \cite{Song12}, we have $X\equiv0$ if $\underline{\sigma
}^{2} <\bar{\sigma}^{2}$. So generally we cannot find a proper process $X$ in
the original $G$-expectation space. Actually, in Theorem \ref{the5.2}, we find
a process $X$ in the extended $\widetilde{G}$-expectation space such that $XY$
is a $\widetilde{G}$-martingale instead of $G$-martingale.

Sometimes we say a process $Y\in S^{\alpha}_{G}(0,T)$ with some $\alpha>1$ is
a solution of equation (\ref{e3}) if there exist processes $Z, K$ such that
$(Y,Z, K)\in\mathfrak{S}_{G}^{\alpha}(0,T)$ is a solution of equation
(\ref{e3}).

\begin{proposition}
\label{mfv} Let $K$ be a decreasing $G$-martingale with $K_{T}\in L^{\alpha
}_{G}(\Omega_{T})$ for some $\alpha>1$. Assume that
\[
f(t, K_{t}, 0)=g(t,K_{t}, 0)=0.
\]
Then $K$ is a solution of equation (\ref{e3}).
\end{proposition}

\begin{proof}
It's easy to check that $(K,0, K)$ is a solution of equation (\ref{e3}).
\end{proof}

\subsection{Comparison theorem of $G$-BSDEs}

\begin{theorem}
\label{the5.5} Let $(Y_{t}^{i},Z_{t}^{i},K_{t}^{i})_{t\leq T}$, $i=1,2$, be
the solutions of the following $G$-BSDEs:%
\[
Y_{t}^{i}=\xi^{i}+\int_{t}^{T}f_{i}(s,Y_{s}^{i},Z_{s}^{i})ds+\int_{t}^{T}%
g_{i}(s,Y_{s}^{i},Z_{s}^{i})d\langle B\rangle_{s}-\int_{t}^{T}Z_{s}^{i}%
dB_{s}-(K_{T}^{i}-K_{t}^{i}),
\]
where $\xi^{i}\in L_{G}^{\beta}(\Omega_{T})$, $f_{i},g_{i}$ satisfy (H1) and
(H2) with $\beta>1$. If $\xi^{1}\geq\xi^{2}$, $f_{1}\geq f_{2}$, $g_{1}\geq
g_{2}$, then $Y_{t}^{1}\geq Y_{t}^{2}$.
\end{theorem}

\begin{proof}
We have%
\[
\hat{Y}_{t}+K_{t}^{2}=\hat{\xi}+K_{T}^{2}+\int_{t}^{T}\hat{f}_{s}ds+\int%
_{t}^{T}\hat{g}_{s}d\langle B\rangle_{s}-\int_{t}^{T}\hat{Z}_{s}dB_{s}%
-(K_{T}^{1}-K_{t}^{1}),
\]
where $\hat{Y}_{t}=Y_{t}^{1}-Y_{t}^{2}$, $\hat{Z}_{t}=Z_{t}^{1}-Z_{t}^{2}$,
$\hat{\xi}=\xi^{1}-\xi^{2}\geq0$, $\hat{f}_{s}=f_{1}(s,Y_{s}^{1},Z_{s}%
^{1})-f_{2}(s,Y_{s}^{2},Z_{s}^{2})$, $\hat{g}_{s}=g_{1}(s,Y_{s}^{1},Z_{s}%
^{1})-g_{2}(s,Y_{s}^{2},Z_{s}^{2})$. For each given $\varepsilon>0$, we can
choose Lipschitz function $l(\cdot)$ such that $I_{[-\varepsilon,\varepsilon
]}\leq l(x)\leq I_{[-2\varepsilon,2\varepsilon]}$. Thus we have%
\[
f_{1}(s,Y_{s}^{1},Z_{s}^{1})-f_{1}(s,Y_{s}^{2},Z_{s}^{1})=(f_{1}(s,Y_{s}%
^{1},Z_{s}^{1})-f_{1}(s,Y_{s}^{2},Z_{s}^{1}))l(\hat{Y}_{s})+a_{s}%
^{\varepsilon}\hat{Y}_{s},
\]
where $a_{s}^{\varepsilon}=(1-l(\hat{Y}_{s}))(f_{1}(s,Y_{s}^{1},Z_{s}%
^{1})-f_{1}(s,Y_{s}^{2},Z_{s}^{1}))\hat{Y}_{s}^{-1}\in M_{G}^{2}(0,T)$ such
that $|a_{s}^{\varepsilon}|\leq L$. It is easy to verify that%
\[
|(f_{1}(s,Y_{s}^{1},Z_{s}^{1})-f_{1}(s,Y_{s}^{2},Z_{s}^{1}))l(\hat{Y}%
_{s})|\leq L|\hat{Y}_{s}|l(\hat{Y}_{s})\leq2L\varepsilon.
\]
Thus we can get%
\[
\hat{f}_{s}=a_{s}^{\varepsilon}\hat{Y}_{s}+b_{s}^{\varepsilon}\hat{Z}%
_{s}+m_{s}-m_{s}^{\varepsilon},\ \hat{g}_{s}=c_{s}^{\varepsilon}\hat{Y}%
_{s}+d_{s}^{\varepsilon}\hat{Z}_{s}+n_{s}-n_{s}^{\varepsilon},
\]
where $|m_{s}^{\varepsilon}|\leq4L\varepsilon$, $|n_{s}^{\varepsilon}%
|\leq4L\varepsilon$, $m_{s}=f_{1}(s,Y_{s}^{2},Z_{s}^{2})-f_{2}(s,Y_{s}%
^{2},Z_{s}^{2})\geq0$ and $n_{s}=g_{1}(s,Y_{s}^{2},Z_{s}^{2})-g_{2}%
(s,Y_{s}^{2},Z_{s}^{2})\geq0$. By Theorem \ref{the5.2}, in the extended space,
we have%
\begin{align*}
&  \hat{Y}_{t}+K_{t}^{2}\\
&  =(X_{t}^{\varepsilon})^{-1}\mathbb{\hat{E}}_{t}^{\tilde{G}}[X_{T}%
^{\varepsilon}(\hat{\xi}+K_{T}^{2})+\int_{t}^{T}(m_{s}-m_{s}^{\varepsilon
}-a_{s}^{\varepsilon}K_{s}^{2})X_{s}^{\varepsilon}ds\\
&  +\int_{t}^{T}(n_{s}-n_{s}^{\varepsilon}-c_{s}^{\varepsilon}K_{s}^{2}%
)X_{s}^{\varepsilon}d\langle B\rangle_{s}]\\
&  \geq(X_{t}^{\varepsilon})^{-1}\mathbb{\hat{E}}_{t}^{\tilde{G}}%
[X_{T}^{\varepsilon}K_{T}^{2}-\int_{t}^{T}(m_{s}^{\varepsilon}+a_{s}%
^{\varepsilon}K_{s}^{2})X_{s}^{\varepsilon}ds-\int_{t}^{T}(n_{s}^{\varepsilon
}+c_{s}^{\varepsilon}K_{s}^{2})X_{s}^{\varepsilon}d\langle B\rangle_{s}]\\
&  \geq(X_{t}^{\varepsilon})^{-1}\{ \mathbb{\hat{E}}_{t}^{\tilde{G}}%
[X_{T}^{\varepsilon}K_{T}^{2}-\int_{t}^{T}a_{s}^{\varepsilon}K_{s}^{2}%
X_{s}^{\varepsilon}ds-\int_{t}^{T}c_{s}^{\varepsilon}K_{s}^{2}X_{s}%
^{\varepsilon}d\langle B\rangle_{s}]\\
&  -\mathbb{\hat{E}}_{t}^{\tilde{G}}[\int_{t}^{T}m_{s}^{\varepsilon}%
X_{s}^{\varepsilon}ds+\int_{t}^{T}n_{s}^{\varepsilon}X_{s}^{\varepsilon
}d\langle B\rangle_{s}]\},
\end{align*}
where $\{X_{t}^{\varepsilon}\}_{t\in\lbrack0,T]}$ is the solution of the
following $\tilde{G}$-SDE:%
\[
X_{t}^{\varepsilon}=1+\int_{0}^{t}a_{s}^{\varepsilon}X_{s}^{\varepsilon
}ds+\int_{0}^{t}c_{s}^{\varepsilon}X_{s}^{\varepsilon}d\langle B\rangle
_{s}+\int_{0}^{t}d_{s}^{\varepsilon}X_{s}^{\varepsilon}dB_{s}+\int_{0}%
^{t}b_{s}^{\varepsilon}X_{s}^{\varepsilon}d\tilde{B}_{s}.
\]
By Theorem \ref{the5.2} and Proposition \ref{mfv}, we get%
\[
(X_{t}^{\varepsilon})^{-1}\mathbb{\hat{E}}_{t}^{\tilde{G}}[X_{T}^{\varepsilon
}K_{T}^{2}-\int_{t}^{T}a_{s}^{\varepsilon}K_{s}^{2}X_{s}^{\varepsilon}%
ds-\int_{t}^{T}c_{s}^{\varepsilon}K_{s}^{2}X_{s}^{\varepsilon}d\langle
B\rangle_{s}]=K_{t}^{2}.
\]
Thus%
\[
\hat{Y}_{t}\geq-4L\varepsilon(X_{t}^{\varepsilon})^{-1}\hat{E}^{\tilde{G}%
}[\int_{t}^{T}|X_{s}^{\varepsilon}|ds+\int_{t}^{T}|X_{s}^{\varepsilon
}|d\langle B\rangle_{s}],
\]
which complete the proof by letting $\varepsilon\rightarrow0$.
\end{proof}

\begin{theorem}
\label{the5.6} Let $(Y_{t}^{i},Z_{t}^{i},K_{t}^{i})_{t\leq T}$, $i=1,2$, be
the solutions of the following $G$-BSDEs:%
\[
Y_{t}^{i}=\xi^{i}+\int_{t}^{T}f_{i}(s)ds+\int_{t}^{T}g_{i}(s)d\langle
B\rangle_{s}+V_{T}^{i}-V_{t}^{i}-\int_{t}^{T}Z_{s}^{i}dB_{s}-(K_{T}^{i}%
-K_{t}^{i}),
\]
where $f_{i}(s)=f_{i}(s,Y_{s}^{i},Z_{s}^{i})$, $g_{i}(s)=g_{i}(s,Y_{s}%
^{i},Z_{s}^{i})$, $\xi^{i}\in L_{G}^{\beta}(\Omega_{T})$, $f_{i},g_{i}$
satisfy (H1) and (H2), $(V_{t}^{i})_{t\leq T}$ are RCLL processes such that
$\mathbb{\hat{E}}[\sup_{t\in\lbrack0,T]}|V_{t}^{i}|^{\beta}]<\infty$ with
$\beta>1$. If $\xi^{1}\geq\xi^{2}$, $f_{1}\geq f_{2}$, $g_{1}\geq g_{2}$,
$V_{t}^{1}-V_{t}^{2}$ is an increasing process, then $Y_{t}^{1}\geq Y_{t}^{2}$.
\end{theorem}

\begin{proof}
The proof is similar to that of Theorem \ref{the5.5}.
\end{proof}

\begin{remark}
\label{rem5.7} If $f_{i},g_{i}$, $i=1,2$, do not contain $Z$, we get the
following special $G$-BSDEs:
\[
Y_{t}^{i}=\mathbb{\hat{E}}_{t}[\xi^{i}+\int_{t}^{T}f_{i}(s,Y_{s}^{i}%
)ds+\int_{t}^{T}g_{i}(s,Y_{s}^{i})d\langle B\rangle_{s}].
\]
The same as in Remark \ref{rem5.4}, here we suppose that $\xi\in L_{G}%
^{1}(\Omega)$, $\{f_{i}(s,y)\}_{s\in\lbrack0,T]}\in M_{G}^{1}(0,T)$ and
$\{g_{i}(s,y)\}_{s\in\lbrack0,T]}\in M_{G}^{1}(0,T)$ for each $y\in\mathbb{R}%
$, $f_{i}$ and $g_{i}$ satisfy the Lipschitz condition with respect to $y$.
The comparison theorem still holds for this case.
\end{remark}

In the following, we give an example to show that the strict comparison
theorem does not hold.

\begin{example}
\label{exa5.8} We consider the simplest $G$-BSDE:
\[
Y_{t}=\xi-\int_{t}^{T}Z_{s}dB_{s}-(K_{T}-K_{t}),
\]
the solution $Y_{t}=\mathbb{\hat{E}}_{t}[\xi]$, $t\in\lbrack0,T]$. Let
$\xi^{1}=0$ and $\xi^{2}=\langle B\rangle_{T}-\bar{\sigma}^{2}T$. It is easy
to verify that $\xi^{1}\geq\xi^{2}$ and $\mathbb{\hat{E}}[\xi^{1}-\xi^{2}]>0$
for the case $\underline{\sigma}<\bar{\sigma}$. But $\mathbb{\hat{E}}[\xi
^{1}]=\mathbb{\hat{E}}[\xi^{2}]=0$.
\end{example}

We now give an application of comparison theorem.

\begin{theorem}
\label{the5.9} (Gronwall inequality) Let $(Y_{t})_{t\leq T}\in S_{G}^{1}(0,T)$
satisfy
\[
Y_{t}\leq\mathbb{\hat{E}}_{t}[\xi+\int_{t}^{T}f(s,Y_{s})ds+\int_{t}%
^{T}g(s,Y_{s})d\langle B\rangle_{s}],
\]
where $\xi\in L_{G}^{1}(\Omega)$, $\{f(s,y)\}_{s\in\lbrack0,T]}\in M_{G}%
^{1}(0,T)$ and $\{g(s,y)\}_{s\in\lbrack0,T]}\in M_{G}^{1}(0,T)$ for each
$y\in\mathbb{R}$, $f$ and $g$ satisfy the Lipschitz condition with respect to
$y$, $f(\cdot,y_{1})\leq f(\cdot,y_{2})$ and $g(\cdot,y_{1})\leq g(\cdot
,y_{2})$ for each $y_{1}\leq y_{2}$. Then $Y_{t}\leq\tilde{Y}_{t}$, where
$(\tilde{Y}_{t})_{t\leq T}$ is the solution of the following $G$-BSDE:%
\[
\tilde{Y}_{t}=\mathbb{\hat{E}}_{t}[\xi+\int_{t}^{T}f(s,\tilde{Y}_{s}%
)ds+\int_{t}^{T}g(s,\tilde{Y}_{s})d\langle B\rangle_{s}].
\]
In particular, if $f(s,y)=a_{s}y+m_{s}$, $g(s,y)=c_{s}y+n_{s}$, where
$a_{s}\geq0$, $c_{s}\geq0$, then
\begin{equation}
Y_{t}\leq(X_{t})^{-1}\mathbb{\hat{E}}_{t}[X_{T}\xi+\int_{t}^{T}m_{s}%
X_{s}ds+\int_{t}^{T}n_{s}X_{s}d\langle B\rangle_{s}], \label{Gronwall}%
\end{equation}
where $X_{t}=\exp(\int_{0}^{t}a_{s}ds+\int_{0}^{t}c_{s}d\langle B\rangle_{s})$.
\end{theorem}

\begin{proof}
We set%
\[
\delta_{t}=\mathbb{\hat{E}}_{t}[\xi+\int_{t}^{T}f(s,Y_{s})ds+\int_{t}%
^{T}g(s,Y_{s})d\langle B\rangle_{s}]-Y_{t}\geq0\text{,}%
\]
then%
\begin{align*}
Y_{t}+\delta_{t}  &  =\mathbb{\hat{E}}_{t}[\xi+\int_{t}^{T}f(s,Y_{s}%
)ds+\int_{t}^{T}g(s,Y_{s})d\langle B\rangle_{s}]\\
&  =\mathbb{\hat{E}}_{t}[\xi+\int_{t}^{T}f(s,Y_{s}+\delta_{s}-\delta
_{s})ds+\int_{t}^{T}g(s,Y_{s}+\delta_{s}-\delta_{s})d\langle B\rangle_{s}].
\end{align*}
Thus $(Y_{t}+\delta_{t})_{t\leq T}$ is the solution of the following $G$-BSDE:%
\[
\bar{Y}_{t}=\mathbb{\hat{E}}_{t}[\xi+\int_{t}^{T}f(s,\bar{Y}_{s}-\delta
_{s})ds+\int_{t}^{T}g(s,\bar{Y}_{s}-\delta_{s})d\langle B\rangle_{s}].
\]
By comparison theorem of $G$-BSDEs, we get $\bar{Y}_{t}\leq\tilde{Y}_{t}$.
Thus $Y_{t}\leq\tilde{Y}_{t}$. By formula (\ref{LBSDE6}) , we get
(\ref{Gronwall}).
\end{proof}

\section{Nonlinear Feynman-Kac Formula}

In this section, we give the nonlinear Feynman-Kac Formula which was studied
in Peng \cite{P10} for special type of $G$-BSDEs. Let $G:\mathbb{S}%
_{d}\rightarrow\mathbb{R}$ be a given monotonic and sublinear function such
that $G(A)-G(B)\geq\underline{\sigma}^{2}\mathrm{tr}[A-B]$ for any $A\geq B$
and $B_{t}=(B_{t}^{i})_{i=1}^{d}$ be the corresponding $G$-Brownian motion. We
consider the following type of $G$-FBSDEs:%
\begin{equation}
dX_{s}^{t,\xi}=b(s,X_{s}^{t,\xi})ds+h_{ij}(s,X_{s}^{t,\xi})d\langle
B^{i},B^{j}\rangle_{s}+\sigma_{j}(s,X_{s}^{t,\xi})dB_{s}^{j},\ X_{t}^{t,\xi
}=\xi, \label{App1}%
\end{equation}

\begin{align}
Y_{s}^{t,\xi}  &  =\Phi(X_{T}^{t,\xi})+\int_{s}^{T}f(r,X_{r}^{t,\xi}%
,Y_{r}^{t,\xi},Z_{r}^{t,\xi})dr+\int_{s}^{T}g_{ij}(r,X_{r}^{t,\xi}%
,Y_{r}^{t,\xi},Z_{r}^{t,\xi})d\langle B^{i},B^{j}\rangle_{r}\nonumber\\
&  -\int_{s}^{T}Z_{r}^{t,\xi}dB_{r}-(K_{T}^{t,\xi}-K_{s}^{t,\xi}),
\label{App2}%
\end{align}
where $b$, $h_{ij}$, $\sigma_{j}:[0,T]\times\mathbb{R}^{n}\rightarrow
\mathbb{R}^{n}$, $\Phi:\mathbb{R}^{n}\rightarrow\mathbb{R}$, $f$, $g_{ij}:$
$[0,T]\times\mathbb{R}^{n}\times\mathbb{R}\times\mathbb{R}^{d}\rightarrow
\mathbb{R}$ are deterministic functions and satisfy the following conditions:

\begin{description}
\item[(A1)] $h_{ij}=h_{ji}$ and $g_{ij}=g_{ji}$ for $1\leq i,j\leq d$;

\item[(A2)] $b$, $h_{ij}$, $\sigma_{j}$, $f$, $g_{ij}$ are continuous in $t$;

\item[(A3)] There exist a positive integer $m$ and a constant $L>0$ such that%
\[
|b(t,x)-b(t,x^{\prime})|+\sum_{i,j=1}^{d}|h_{ij}(t,x)-h_{ij}(t,x^{\prime
})|+\sum_{j=1}^{d}|\sigma_{j}(t,x)-\sigma_{j}(t,x^{\prime})|\leq
L|x-x^{\prime}|,
\]%
\[
|\Phi(x)-\Phi(x^{\prime})|\leq L(1+|x|^{m}+|x^{\prime}|^{m})|x-x^{\prime}|,
\]
\begin{align*}
&  |f(t,x,y,z)-f(t,x^{\prime},y^{\prime},z^{\prime})|+\sum_{i,j=1}^{d}%
|g_{ij}(t,x,y,z)-g_{ij}(t,x^{\prime},y^{\prime},z^{\prime})|\\
&  \leq L[(1+|x|^{m}+|x^{\prime}|^{m})|x-x^{\prime}|+|y-y^{\prime
}|+|z-z^{\prime}|].
\end{align*}

\end{description}

We have the following estimates of $G$-SDEs which can be found in Chapter V in
Peng \cite{P10}.

\begin{proposition}
\label{proA.1} Let $\xi$, $\xi^{\prime}\in L_{G}^{p}(\Omega_{t};\mathbb{R}%
^{n})$ with $p\geq2$. Then we have, for each $\delta\in\lbrack0,T-t]$,%
\[
\mathbb{\hat{E}}_{t}[|X_{t+\delta}^{t,\xi}-X_{t+\delta}^{t,\xi^{\prime}}%
|^{p}]\leq C|\xi-\xi^{\prime}|^{p},
\]%
\[
\mathbb{\hat{E}}_{t}[|X_{t+\delta}^{t,\xi}|^{p}]\leq C(1+|\xi|^{p}),
\]%
\[
\mathbb{\hat{E}}_{t}[\sup_{s\in\lbrack t,t+\delta]}|X_{s}^{t,\xi}-\xi
|^{p}]\leq C(1+|\xi|^{p})\delta^{p/2},
\]
where the constant $C$ depends on $L$, $G$, $p$, $n$ and $T$.
\end{proposition}

\begin{proof}
For convenience of the reader, we sketch the proof. It is easy to verify that
$(X_{s}^{t,\xi})_{s\in\lbrack t,T]}$, $(X_{s}^{t,\xi^{\prime}})_{s\in\lbrack
t,T]}\in M_{G}^{p}(0,T;\mathbb{R}^{n})$. By Remark \ref{nnrem1}, we can get%
\begin{align*}
\mathbb{\hat{E}}_{t}[|X_{t+\delta}^{t,\xi}-X_{t+\delta}^{t,\xi^{\prime}}%
|^{p}]  &  \leq C_{1}(|\xi-\xi^{\prime}|^{p}+\mathbb{\hat{E}}_{t}[\int%
_{t}^{t+\delta}|X_{s}^{t,\xi}-X_{s}^{t,\xi^{\prime}}|^{p}ds])\\
&  \leq C_{1}(|\xi-\xi^{\prime}|^{p}+\int_{t}^{t+\delta}\mathbb{\hat{E}}%
_{t}[|X_{s}^{t,\xi}-X_{s}^{t,\xi^{\prime}}|^{p}]ds),
\end{align*}
where the constant $C_{1}$ depends on $L$, $G$, $p$, $n$ and $T$. By the
Gronwall inequality, we obtain%
\[
\mathbb{\hat{E}}_{t}[|X_{t+\delta}^{t,\xi}-X_{t+\delta}^{t,\xi^{\prime}}%
|^{p}]\leq C_{1}\exp(C_{1}T)|\xi-\xi^{\prime}|^{p}.
\]
Then we get the first inequality. The other inequalities can be proved similarly.
\end{proof}

\begin{proposition}
\label{proA.2} For each $\xi$, $\xi^{\prime}\in L_{G}^{4m+1}(\Omega
_{t};\mathbb{R}^{n})$, we have%
\[
|Y_{t}^{t,\xi}-Y_{t}^{t,\xi^{\prime}}|\leq C(1+|\xi|^{m}+|\xi^{\prime}%
|^{m})|\xi-\xi^{\prime}|,
\]%
\[
|Y_{t}^{t,\xi}|\leq C(1+|\xi|^{m+1}),
\]
where the constant $C$ depends on $L$, $G$, $n$ and $T$.
\end{proposition}

\begin{proof}
It follows from Proposition \ref{pro3.5} and Proposition \ref{proA.1} that%
\begin{align*}
|Y_{t}^{t,\xi}-Y_{t}^{t,\xi^{\prime}}|^{2}  &  \leq C_{1}\{ \mathbb{\hat{E}%
}_{t}[(1+|X_{T}^{t,\xi}|^{m}+|X_{T}^{t,\xi^{\prime}}|^{m})^{2}|X_{T}^{t,\xi
}-X_{T}^{t,\xi^{\prime}}|^{2}]\\
&  +\int_{t}^{T}\mathbb{\hat{E}}_{t}[(1+|X_{s}^{t,\xi}|^{m}+|X_{s}%
^{t,\xi^{\prime}}|^{m})^{2}|X_{s}^{t,\xi}-X_{s}^{t,\xi^{\prime}}|^{2}]ds\}\\
&  \leq C_{2}(1+|\xi|^{2m}+|\xi^{\prime}|^{2m})\{(\mathbb{\hat{E}}_{t}%
[|X_{T}^{t,\xi}-X_{T}^{t,\xi^{\prime}}|^{4}])^{1/2}\\
&  +\int_{t}^{T}(\mathbb{\hat{E}}_{t}[|X_{s}^{t,\xi}-X_{s}^{t,\xi^{\prime}%
}|^{4}])^{1/2}ds\}\\
&  \leq C_{3}(1+|\xi|^{2m}+|\xi^{\prime}|^{2m})|\xi-\xi^{\prime}|^{2},
\end{align*}
where $C_{1}$, $C_{2}$ and $C_{3}$ depend on $L$, $G$, $n$ and $T$. Thus we
get $|Y_{t}^{t,\xi}-Y_{t}^{t,\xi^{\prime}}|\leq C(1+|\xi|^{m}+|\xi^{\prime
}|^{m})|\xi-\xi^{\prime}|$. By Proposition \ref{pro3.4}, we can get
$|Y_{t}^{t,\xi}|\leq C(1+|\xi|^{m+1})$ by using the similar analysis.
\end{proof}

We are more interested in the case when $\xi=x\in\mathbb{R}^{n}$. We define%
\[
u(t,x):=Y_{t}^{t,x},\ \ (t,x)\in\lbrack0,T]\times\mathbb{R}^{n}.
\]
By Proposition \ref{proA.2}, we immediately have the following estimates:%
\[
|u(t,x)-u(t,x^{\prime})|\leq C(1+|x|^{m}+|x^{\prime}|^{m})|x-x^{\prime}|,
\]%
\[
|u(t,x)|\leq C(1+|x|^{m+1}),
\]
where the constant $C$ depends on $L$, $G$, $n$ and $T$.

\begin{remark}
\label{remA.3} It is important to note that $u(t,x)$ is a deterministic
function of $(t,x)$, because $b$, $h_{ij}$, $\sigma_{j}$, $\Phi$, $f$,
$g_{ij}$ are deterministic functions and $\tilde{B}_{s}:=B_{t+s}-B_{t}$ is a
$G$-Brownian motion.
\end{remark}

The following theorem plays a key role in proving the Feynman-Kac formula.

\begin{theorem}
\label{theA.4} For each $\xi\in L_{G}^{4m+1}(\Omega_{t};\mathbb{R}^{n})$, we
have%
\[
u(t,\xi)=Y_{t}^{t,\xi}.
\]

\end{theorem}

\begin{proof}
By Proposition \ref{proA.2}, we only need to prove Theorem \ref{theA.4} for
bounded $\xi\in L_{G}^{4m+1}(\Omega_{t};\mathbb{R}^{n})$. Thus for each
$\varepsilon>0$, we can choose a simple function
\[
\eta^{\varepsilon}=\sum_{i=1}^{N}x_{i}I_{A_{i}},
\]
where $(A_{i})_{i=1}^{N}$ is a $\mathcal{B}(\Omega_{t})$-partition and
$x_{i}\in\mathbb{R}^{n}$, such that $|\eta^{\varepsilon}-\xi|\leq\varepsilon$.
It follows from Proposition \ref{proA.2} that%
\begin{align*}
|Y_{t}^{t,\xi}-u(t,\eta^{\varepsilon})|  &  =|Y_{t}^{t,\xi}-\sum_{i=1}%
^{n}u(t,x_{i})I_{A_{i}}|\\
&  =|Y_{t}^{t,\xi}-\sum_{i=1}^{N}Y_{t}^{t,x_{i}}I_{A_{i}}|\\
&  =\sum_{i=1}^{N}|Y_{t}^{t,\xi}-Y_{t}^{t,x_{i}}|I_{A_{i}}\\
&  \leq\sum_{i=1}^{N}C(1+|\xi|^{m})|\xi-x_{i}|I_{A_{i}}\\
&  =C(1+|\xi|^{m})|\xi-\sum_{i=1}^{N}x_{i}I_{A_{i}}|\\
&  \leq C(1+|\xi|^{m})\varepsilon,
\end{align*}
where the constant $C$ depends on $L$, $G$, $n$ and $T$. Noting that
\[
|u(t,\xi)-u(t,\eta^{\varepsilon})|\leq C(1+|\xi|^{m})|\xi-\eta^{\varepsilon
}|\leq C(1+|\xi|^{m})\varepsilon,
\]
we get $|Y_{t}^{t,\xi}-u(t,\xi)|\leq2C(1+|\xi|^{m})\varepsilon$. Since
$\varepsilon$ can be arbitrarily small, we obtain $Y_{t}^{t,\xi}=u(t,\xi)$.
\end{proof}

We now give the Feynman-Kac formula.

\begin{theorem}
\label{theA.9} Let $u(t,x):=Y_{t}^{t,x}$ for $(t,x)\in\lbrack0,T]\times
\mathbb{R}^{n}$. Then $u(t,x)$ is the unique viscosity solution of the
following PDE:%
\begin{equation}
\left\{
\begin{array}
[c]{l}%
\partial_{t}u+F(D_{x}^{2}u,D_{x}u,u,x,t)=0,\\
u(T,x)=\Phi(x),
\end{array}
\right.  \label{feynman}%
\end{equation}
where%
\begin{align*}
F(D_{x}^{2}u,D_{x}u,u,x,t)=  &  G(H(D_{x}^{2}u,D_{x}u,u,x,t))+\langle
b(t,x),D_{x}u\rangle\\
&  +f(t,x,u,\langle\sigma_{1}(t,x),D_{x}u\rangle,\ldots,\langle\sigma
_{d}(t,x),D_{x}u\rangle),
\end{align*}%
\begin{align*}
H_{ij}(D_{x}^{2}u,D_{x}u,u,x,t)=  &  \langle D_{x}^{2}u\sigma_{i}%
(t,x),\sigma_{j}(t,x)\rangle+2\langle D_{x}u,h_{ij}(t,x)\rangle\\
&  +2g_{ij}(t,x,u,\langle\sigma_{1}(t,x),D_{x}u\rangle,\ldots,\langle
\sigma_{d}(t,x),D_{x}u\rangle).
\end{align*}

\end{theorem}

\begin{proof}
The uniqueness of viscosity solution of equation (\ref{feynman}) can be found
in Appendix C in Peng \cite{P10}, we only prove that $u$ is a viscosity
solution of equation (\ref{feynman}). By $Y_{t+\delta}^{t,x}=Y_{t+\delta
}^{t+\delta,X_{t+\delta}^{t,x}}$ and Theorem \ref{theA.4}, we get
$Y_{t+\delta}^{t,x}=u(t+\delta,X_{t+\delta}^{t,x})$ for $\delta\in
\lbrack0,T-t]$ and
\begin{align*}
Y_{t}^{t,x}=  &  u(t+\delta,X_{t+\delta}^{t,x})+\int_{t}^{t+\delta}%
f(r,X_{r}^{t,x},Y_{r}^{t,x},Z_{r}^{t,x})dr\\
&  +\int_{t}^{t+\delta}g_{ij}(r,X_{r}^{t,x},Y_{r}^{t,x},Z_{r}^{t,x})d\langle
B^{i},B^{j}\rangle_{r}-\int_{t}^{t+\delta}Z_{r}^{t,x}dB_{r}-(K_{t+\delta
}^{t,x}-K_{t}^{t,x}).
\end{align*}
Taking $G$-expectation, we get%
\[
u(t,x)=\mathbb{\hat{E}}[u(t+\delta,X_{t+\delta}^{t,x})+\int_{t}^{t+\delta
}f_{r}dr+\int_{t}^{t+\delta}g_{r}^{ij}d\langle B^{i},B^{j}\rangle_{r}],
\]
where $f_{r}=f(r,X_{r}^{t,x},Y_{r}^{t,x},Z_{r}^{t,x})$, $g_{r}^{ij}%
=g_{ij}(r,X_{r}^{t,x},Y_{r}^{t,x},Z_{r}^{t,x})$. In order to prove that $u$ is
a viscosity solution, we first show that $u$ is a continuous function. By
Proposition \ref{proA.2}, we know that $|u(t,x)-u(t,x^{\prime})|\leq
C(1+|x|^{m}+|x^{\prime}|^{m})|x-x^{\prime}|$. By Proposition \ref{proA.1} and
Proposition \ref{pro3.4}, we have $\mathbb{\hat{E}}_{t}[|X_{t+\delta}%
^{t,x}-x|^{2}]\leq C(1+|x|^{2})\delta$ and $\mathbb{\hat{E}}_{t}[|Y_{r}%
^{t,x}|^{2}+\int_{t}^{T}|Z_{r}^{t,x}|^{2}dr]\leq C(1+|x|^{2m+2})$, where $C$
depends on $L$, $G$, $n$ and $T$. \ Thus we get
\begin{align*}
&  |u(t,x)-u(t+\delta,x)|\\
&  \leq C\{(1+|x|^{m})(\mathbb{\hat{E}}[|X_{t+\delta}^{t,x}-x|^{2}%
])^{1/2}+(\mathbb{\hat{E}}[\int_{t}^{T}(|f_{r}|^{2}+|g_{r}^{ij}|^{2}%
)dr])^{1/2}\delta^{1/2}\}\\
&  \leq C(1+|x|^{m+1})\delta^{1/2}.
\end{align*}
It follows that $u$ is a continuous function. For any fixed $(t,x)\in
(0,T)\times\mathbb{R}^{n}$, let $\psi\in C^{2,3}([0,T]\times\mathbb{R}^{n})$
be such that $\psi\geq u$, $\psi(t,x)=u(t,x)$ and $|\partial_{tx_{i}}^{2}%
\psi(t,x)|+|\partial_{x_{i}}\psi(t,x)|+|\partial_{x_{i}x_{j}}^{2}%
\psi(t,x)|+|\partial_{x_{i}x_{j}x_{k}}^{3}\psi(t,x)|\leq C(1+|x|^{m_{1}})$ for
some $m_{1}>0$. Let $(\tilde{Y},\tilde{Z},\tilde{K})$ be the solution of
$G$-BSDE (\ref{App2}) on $[t,t+\delta]$ with terminal condition $\psi
(t+\delta,X_{t+\delta}^{t,x})$. Set $\hat{Y}_{s}^{1}=\tilde{Y}_{s}%
-\psi(s,X_{s}^{t,x})$, $\hat{Z}_{s}^{1}=\tilde{Z}_{s}-(\langle\sigma
_{1}(s,X_{s}^{t,x}),D_{x}\psi(s,X_{s}^{t,x})\rangle,\cdots,\langle\sigma
_{d}(s,X_{s}^{t,x}),D_{x}\psi(s,X_{s}^{t,x})\rangle)$, $\hat{K}_{s}^{1}%
=\tilde{K}_{s}$, applying It\^{o}'s formula to $\tilde{Y}_{s}-\psi
(s,X_{s}^{t,x})$, we obtain that $(\hat{Y}^{1},\hat{Z}^{1},\hat{K}^{1})$ is
the solution of the following $G$-BSDE:
\begin{align*}
\hat{Y}_{s}^{1}=  &  \int_{s}^{t+\delta}F_{1}(r,X_{r}^{t,x},\hat{Y}_{r}%
^{1},\hat{Z}_{r}^{1})dr+\int_{s}^{t+\delta}F_{2}^{ij}(r,X_{r}^{t,x},\hat
{Y}_{r}^{1},\hat{Z}_{r}^{1})d\langle B^{i},B^{j}\rangle_{r}\\
&  -\int_{s}^{t+\delta}\hat{Z}_{r}^{1}dB_{r}-(\hat{K}_{t+\delta}^{1}-\hat
{K}_{s}^{1}),
\end{align*}
where%
\begin{align*}
F_{1}(r,x,y,z)  &  =f(r,x,y+\psi(r,x),z+(\langle\sigma_{1},D_{x}\psi
\rangle,\cdots,\langle\sigma_{d},D_{x}\psi\rangle)(r,x))\\
&  +\partial_{t}\psi(r,x)+\langle b(r,x),D_{x}\psi(r,x)\rangle,
\end{align*}%
\begin{align*}
F_{2}^{ij}(r,x,y,z)  &  =g_{ij}(r,x,y+\psi(r,x),z+(\langle\sigma_{1},D_{x}%
\psi\rangle,\cdots,\langle\sigma_{d},D_{x}\psi\rangle)(r,x))\\
&  +\langle D_{x}\psi(r,x),h_{ij}(r,x)\rangle+\frac{1}{2}\langle D_{x}^{2}%
\psi(r,x)\sigma_{i}(r,x),\sigma_{j}(r,x)\rangle.
\end{align*}
Let $(\hat{Y},\hat{Z},\hat{K})$ be the solution of the following $G$-BSDE:%
\begin{align*}
\hat{Y}_{s}=  &  \int_{s}^{t+\delta}F_{1}(r,x,\hat{Y}_{r},\hat{Z}_{r}%
)dr+\int_{s}^{t+\delta}F_{2}^{ij}(r,x,\hat{Y}_{r},\hat{Z}_{r})d\langle
B^{i},B^{j}\rangle_{r}\\
&  -\int_{s}^{t+\delta}\hat{Z}_{r}dB_{r}-(\hat{K}_{t+\delta}-\hat{K}_{s}).
\end{align*}
It is easy to check that $\hat{Z}_{s}=0$, $\hat{Y}_{s}$ is the solution of the
following ODE:%
\[
\hat{Y}_{s}=\int_{s}^{t+\delta}[F_{1}(r,x,\hat{Y}_{r},0)+2G(F_{2}(r,x,\hat
{Y}_{r},0))]dr,
\]%
\[
\hat{K}_{s}=\int_{t}^{s}F_{2}^{ij}(r,x,\hat{Y}_{r},0)d\langle B^{i}%
,B^{j}\rangle_{r}-\int_{t}^{s}2G(F_{2}(r,x,\hat{Y}_{r},0))dr,
\]
where $F_{2}(r,x,\hat{Y}_{r},0)=(F_{2}^{ij}(r,x,\hat{Y}_{r},0))_{i,j=1}^{d}$.
By Proposition \ref{pro3.5}, we have for any fixed $p>2$%
\begin{align*}
|\hat{Y}_{t}^{1}-\hat{Y}_{t}|^{2}  &  \leq\mathbb{\hat{E}}[\sup_{s\in\lbrack
t,t+\delta]}|\hat{Y}_{s}^{1}-\hat{Y}_{s}|^{2}]\\
&  \leq C\{(\mathbb{\hat{E}}[\sup_{s\in\lbrack t,t+\delta]}\mathbb{\hat{E}%
}_{s}[(\int_{t}^{t+\delta}\hat{F}_{r}dr)^{p}]])^{2/p}+\mathbb{\hat{E}}%
[\sup_{s\in\lbrack t,t+\delta]}\mathbb{\hat{E}}_{s}[(\int_{t}^{t+\delta}%
\hat{F}_{r}dr)^{p}]]\},
\end{align*}
where $\hat{F}_{r}=|F_{1}(r,X_{r}^{t,x},\hat{Y}_{r},0)-F_{1}(r,x,\hat{Y}%
_{r},0)|+\sum_{i,j=1}^{d}|F_{2}^{ij}(r,X_{r}^{t,x},\hat{Y}_{r},0)-F_{2}%
^{ij}(r,x,\hat{Y}_{r},0)|$. It is easy to verify that there exists a constant
$m_{2}>0$ such that
\[
\hat{F}_{r}\leq C(1+|x|^{m_{2}}+|X_{r}^{t,x}|^{m_{2}})|X_{r}^{t,x}-x|.
\]
Then by Theorem 2.13 in \cite{HJPS} and Proposition \ref{proA.1} we can deduce
that $|\hat{Y}_{t}^{1}-\hat{Y}_{t}|\leq C(1+|x|^{m_{2}+2})\delta^{\frac{3}{2}%
}$. By comparison theorem of $G$-BSDEs, we know that $\tilde{Y}_{t}\geq
u(t,x)$, that is $\hat{Y}_{t}^{1}\geq0$. Then we get%
\[
-C(1+|x|^{m_{2}+2})\delta^{1/2}\leq\delta^{-1}\hat{Y}_{t}=\delta^{-1}\int%
_{t}^{t+\delta}[F_{1}(r,x,\hat{Y}_{r},0)+2G(F_{2}(r,x,\hat{Y}_{r},0))]dr.
\]
Letting $\delta\rightarrow0$, we obtain $F_{1}(t,x,0,0)+2G(F_{2}%
(t,x,0,0))\geq0$, which implies that $u$ is a viscosity subsolution. Similarly
we can prove that $u$ is a viscosity supersolution.
\end{proof}

\section{Girsanov transformation}

\subsection{Nonlinear expectations generated by $G$-BSDEs}

For simplicity, we consider the following $G$-BSDE driven by $1$-dimensional
$G$-Brownian motion. The results still hold for the case $d>1$.
\begin{align}
Y_{t}^{T,\xi}=  &  \xi+\int_{t}^{T}f(s,Y_{s}^{T,\xi},Z_{s}^{T,\xi})ds+\int%
_{t}^{T}g(s,Y_{s}^{T,\xi},Z_{s}^{T,\xi})d\langle B\rangle_{s}\nonumber\\
&  -\int_{t}^{T}Z_{s}^{T,\xi}dB_{s}-(K_{T}^{T,\xi}-K_{t}^{T,\xi}),
\label{NonExp}%
\end{align}
where $f$ and $g$ satisfy the Lipschitz condition. We further suppose that
$f(s,y,0)=g(s,y,0)=0$. We define, for each $\xi\in L_{G}^{\beta}(\Omega_{T})$
with $\beta>1$,
\[
\mathbb{\tilde{E}}_{t,T}[\xi]:=Y_{t}^{T,\xi}.
\]
It is easy to verify that for each $T_{1}<T_{2}$ and $\xi\in L_{G}^{\beta
}(\Omega_{T_{1}})$ with $\beta>1$, $\mathbb{\tilde{E}}_{t,T_{1}}%
[\xi]=\mathbb{\tilde{E}}_{t,T_{2}}[\xi]$. Thus we use the notation
$\mathbb{\tilde{E}}_{t}[\xi]$.

\begin{theorem}
\label{the6.1} We have

\begin{description}
\item[(1)] For each $\xi^{1}\geq\xi^{2}$, we have $\mathbb{\tilde{E}}_{t}%
[\xi^{1}]\geq\mathbb{\tilde{E}}_{t}[\xi^{2}]$;

\item[(2)] For each $\xi\in L_{G}^{\beta}(\Omega_{t})$ with $\beta>1$,
$\mathbb{\tilde{E}}_{t}[\xi]=\xi$;

\item[(3)] $\mathbb{\tilde{E}}_{t}[\mathbb{\tilde{E}}_{s}[\xi]]=\mathbb{\tilde
{E}}_{t\wedge s}[\xi]$;

\item[(4)] If $f$ and $g$ are positively homogeneous, then for each
$\lambda_{t}\in L_{G}^{\infty}(\Omega_{t})$, we have $\mathbb{\tilde{E}}%
_{t}[\lambda_{t}\xi]=\lambda_{t}\mathbb{\tilde{E}}_{t}[\xi]$;

\item[(5)] If $f$ and $g$ are subadditive, then $\mathbb{\tilde{E}}_{t}%
[\xi^{1}+\xi^{2}]\leq\mathbb{\tilde{E}}_{t}[\xi^{1}]+\mathbb{\tilde{E}}%
_{t}[\xi^{2}]$;

\item[(6)] If $f$ and $g$ are convex, then $\mathbb{\tilde{E}}_{t}[\lambda
_{t}\xi^{1}+(1-\lambda_{t})\xi^{2}]\leq\lambda_{t}\mathbb{\tilde{E}}_{t}%
[\xi^{1}]+(1-\lambda_{t})\mathbb{\tilde{E}}_{t}[\xi^{2}]$ for each
$\lambda_{t}\in L_{G}^{\infty}(\Omega_{t})$ and $\lambda_{t}\in\lbrack0,1]$;

\item[(7)] For each $\xi\in L_{G}^{1}(\Omega_{t};\mathbb{R}^{m})$, $\eta\in
L_{G}^{1}(\Omega_{T};\mathbb{R}^{n})$, $\Phi\in C_{b.Lip}(\mathbb{R}^{m+n})$,
we have
\[
\mathbb{\tilde{E}}_{t}[\Phi(\xi,\eta)]=\mathbb{\tilde{E}}_{t}[\Phi
(x,\eta)]_{x=\xi}.
\]

\item[(8)] Let $K$ be a decreasing $G$-martingale with $K_{T}\in L^{\alpha
}_{G}(\Omega_{T})$ for some $\alpha>1$. Then we have
\[
\mathbb{\tilde{E}}_{s}[K_{t}]=K_{s}, \ \textmd{for any} \ s\leq t.
\]

\end{description}
\end{theorem}

\begin{proof}
It is easy to get (1)-(3). (8) is straightforward from Proposition
\ref{mfv}. First we prove (6). (4) and (5) can be proved similarly.
Let $(Y^{i},Z^{i},K^{i})$, $i=1,2$, be the solutions of $G$-BSDE
(\ref{NonExp}) corresponding to $\xi
^{i}$. We have for $r\in\lbrack t,T]$%
\[
\tilde{Y}_{r}=\tilde{\xi}+\int_{r}^{T}\tilde{f}_{s}ds+\int_{r}^{T}\tilde
{g}_{s}d\langle B\rangle_{s}-\tilde{K}_{T}^{2}+\tilde{K}_{r}^{2}-\int_{r}%
^{T}\tilde{Z}_{s}dB_{s}-(\tilde{K}_{T}^{1}-\tilde{K}_{r}^{1}),
\]
where $\tilde{Y}_{r}=\lambda_{t}Y_{r}^{1}+(1-\lambda_{t})Y_{r}^{2}$,
$\tilde{\xi}=\lambda_{t}\xi^{1}+(1-\lambda_{t})\xi^{2}$, $\tilde{f}%
_{s}=\lambda_{t}f(s,Y_{s}^{1},Z_{s}^{1})+(1-\lambda_{t})f(s,Y_{s}^{2}%
,Z_{s}^{2})$, $\tilde{g}_{s}=\lambda_{t}g(s,Y_{s}^{1},Z_{s}^{1})+(1-\lambda
_{t})g(s,Y_{s}^{2},Z_{s}^{2})$, $\tilde{Z}_{s}=\lambda_{t}Z_{s}^{1}%
+(1-\lambda_{t})Z_{s}^{2}$, $\tilde{K}_{r}^{1}=\lambda_{t}K_{r}^{1}$,
$\tilde{K}_{r}^{2}=(1-\lambda_{t})\tilde{K}_{r}^{2}$. By the convexity of $f$
and $g$, we get $\tilde{f}_{s}\geq f(s,\tilde{Y}_{s},\tilde{Z}_{s})$ and
$\tilde{g}_{s}\geq g(s,\tilde{Y}_{s},\tilde{Z}_{s})$. Note that $-\tilde
{K}_{r}$ is an increasing process, then by Theorem \ref{the5.6} we obtain
$\mathbb{\tilde{E}}_{t}[\tilde{\xi}]\leq\tilde{Y}_{t}$, which implies (6).

We now prove (7). For each given $n\in\mathbb{N}$, we can choose $A_{i}^{n}%
\in\mathcal{B}(\mathbb{R}^{m})$, $i=1,\ldots,k_{n}$, such that $A_{i}^{n}\cap
A_{j}^{n}=\varnothing$ for $i\not =j$, $\cup_{i=1}^{k_{n}}A_{i}^{n}%
=\mathbb{R}^{m}$, $\{x:|x|\leq n\} \subset\cup_{i=1}^{k_{n}-1}A_{i}^{n}$ and
$\lambda(A_{i}^{n})\leq1/n$ for $i\leq k_{n}-1$, where $\lambda(A_{i}^{n})$
denote the diameter of $A_{i}$. Let $x_{i}^{n}\in A_{i}^{n}$, by Proposition
\ref{pro3.5}, we have%
\begin{align*}
&  |\sum_{i=1}^{k_{n}}\mathbb{\tilde{E}}_{t}[\Phi(x_{i}^{n},\eta)]I_{A_{i}%
^{n}}(\xi)-\mathbb{\tilde{E}}_{t}[\Phi(\xi,\eta)]|^{2}\\
&  =\sum_{i=1}^{k_{n}}I_{A_{i}^{n}}(\xi)|\mathbb{\tilde{E}}_{t}[\Phi(x_{i}%
^{n},\eta)]-\mathbb{\tilde{E}}_{t}[\Phi(\xi,\eta)]|^{2}\\
&  \leq C\sum_{i=1}^{k_{n}}I_{A_{i}^{n}}(\xi)\mathbb{\hat{E}}_{t}[|\Phi
(x_{i}^{n},\eta)-\Phi(\xi,\eta)|^{2}]\\
&  =C\mathbb{\hat{E}}_{t}[\sum_{i=1}^{k_{n}}I_{A_{i}^{n}}(\xi)|\Phi(x_{i}%
^{n},\eta)-\Phi(\xi,\eta)|^{2}],
\end{align*}
where $C$ is a constant independent of $n$. Note that%
\[
\sum_{i=1}^{k_{n}}I_{A_{i}^{n}}(\xi)|\Phi(x_{i}^{n},\eta)-\Phi(\xi,\eta
)|^{2}\leq\frac{L^{2}}{n^{2}}+4||\Phi||_{\infty}^{2}I_{[|\xi|>n]},
\]
where $L$ is the Lipschitz constant of $\Phi$, then we get%
\begin{align*}
&  \mathbb{\hat{E}}[|\sum_{i=1}^{k_{n}}\mathbb{\tilde{E}}_{t}[\Phi(x_{i}%
^{n},\eta)]I_{A_{i}^{n}}(\xi)-\mathbb{\tilde{E}}_{t}[\Phi(\xi,\eta)]|^{2}]\\
&  \leq C\mathbb{\hat{E}}[\frac{L^{2}}{n^{2}}+4||\Phi||_{\infty}^{2}%
I_{[|\xi|>n]}]\\
&  \leq C\{ \frac{L^{2}}{n^{2}}+\frac{4||\Phi||_{\infty}^{2}}{n}%
\mathbb{\hat{E}}[|\xi|]\} \rightarrow0.
\end{align*}
On the other hand, by Proposition \ref{pro3.5},, we know that there exists a
constant $C>0$ such that%
\[
|\mathbb{\tilde{E}}_{t}[\Phi(x,\eta)]-\mathbb{\tilde{E}}_{t}[\Phi
(y,\eta)]|\leq C|x-y|\text{ \ for }x,y\in\mathbb{R}^{m}.
\]
Thus%
\begin{align*}
&  \mathbb{\hat{E}}[|\sum_{i=1}^{k_{n}}\mathbb{\tilde{E}}_{t}[\Phi(x_{i}%
^{n},\eta)]I_{A_{i}^{n}}(\xi)-\mathbb{\tilde{E}}_{t}[\Phi(x,\eta)]_{x=\xi
}|^{2}]\\
&  =\mathbb{\hat{E}}[\sum_{i=1}^{k_{n}}I_{A_{i}^{n}}(\xi)|\mathbb{\tilde{E}%
}_{t}[\Phi(x_{i}^{n},\eta)]-\mathbb{\tilde{E}}_{t}[\Phi(x,\eta)]_{x=\xi}%
|^{2}]\\
&  \leq\mathbb{\hat{E}}[\frac{C^{2}}{n^{2}}+4||\Phi||_{\infty}^{2}%
I_{[|\xi|>n]}]\\
&  \leq\frac{C^{2}}{n^{2}}+\frac{4||\Phi||_{\infty}^{2}}{n}\mathbb{\hat{E}%
}[|\xi|]\rightarrow0,
\end{align*}
which implies $\mathbb{\tilde{E}}_{t}[\Phi(\xi,\eta)]=\mathbb{\tilde{E}}%
_{t}[\Phi(x,\eta)]_{x=\xi}$.
\end{proof}

\subsection{Girsanov transformation}

We first consider the following $G$-BSDE driven by $1$-dimensional
$G$-Brownian motion:%
\[
Y_{t}=\xi+\int_{t}^{T}b_{s}Z_{s}ds+\int_{t}^{T}d_{s}Z_{s}d\langle B\rangle
_{s}-\int_{t}^{T}Z_{s}dB_{s}-(K_{T}-K_{t}),
\]
where $(b_{t})_{t\leq T}$ and $(d_{t})_{t\leq T}$ are bounded processes. For
each $\xi\in L_{G}^{\beta}(\Omega_{T})$ with $\beta>1$, define%
\[
\mathbb{\tilde{E}}_{t}[\xi]=Y_{t}.
\]
By Theorem \ref{the6.1}, we know that $\mathbb{\tilde{E}}_{t}[\cdot]$ is a
consistent sublinear expectation.

\begin{theorem}
\label{the6.2} (Girsanov Theorem) Let $(b_{t})_{t\leq T}$ and $(d_{t})_{t\leq
T}$ be bounded processes. Then $\bar{B}_{t}:=B_{t}-\int_{0}^{t}b_{s}%
ds-\int_{0}^{t}d_{s}d\langle B\rangle_{s}$ is a $G$-Brownian motion under
$\mathbb{\tilde{E}}$.
\end{theorem}

\begin{proof}
We only need to show that for each $\Phi\in C_{b.Lip}(\mathbb{R}^{n})$,
$t_{1}<\cdots<t_{n}$,
\[
\mathbb{\tilde{E}}[\Phi(\bar{B}_{t_{1}},\bar{B}_{t_{2}}-\bar{B}_{t_{1}}%
,\ldots,\bar{B}_{t_{n}}-\bar{B}_{t_{n-1}})]=\mathbb{\hat{E}}[\Phi(B_{t_{1}%
},B_{t_{2}}-B_{t_{1}},\ldots,B_{t_{n}}-B_{t_{n-1}})].
\]

Step 1. We consider the case $b_{s}\equiv b$ and $d_{s}\equiv d$. For each
$\varphi\in C_{b.Lip}(\mathbb{R})$, we define%
\[
\tilde{u}(t,x)=\mathbb{\tilde{E}}[\varphi(x+\bar{B}_{t})].
\]
Set $u(t,x)=\tilde{u}(T-t,x)$ for fixed $T>0$, by Theorem \ref{theA.9}, we
obtain $u$ satisfies the following PDE:%
\[
\partial_{t}u-b\partial_{x}u+b\partial_{x}u+2G(-d\partial_{x}u+\frac{1}%
{2}\partial_{xx}^{2}u+d\partial_{x}u)=0,u(T,x)=\varphi(x),
\]
i.e. $\partial_{t}u+G(\partial_{xx}^{2}u)=0$, $u(T,x)=\varphi(x)$. Thus
$\mathbb{\tilde{E}}[\varphi(\bar{B}_{t})]=\mathbb{\hat{E}}[\varphi(B_{t})]$
for any $t\geq0$, $\varphi\in C_{b.Lip}(\mathbb{R})$.

Step 2. We consider the case $b_{s}^{n}=\sum_{i=0}^{n-1}\xi_{i}I_{[t_{i}%
^{n},t_{i+1}^{n})}(s)$, $d_{s}^{n}=\sum_{i=0}^{n-1}\eta_{i}I_{[t_{i}%
^{n},t_{i+1}^{n})}(s)$, where $\xi_{i}$, $\eta_{i}\in Lip(\Omega_{t_{i}^{n}}%
)$. For each $\varphi\in C_{b.Lip}(\mathbb{R})$, we have%
\[
\mathbb{\tilde{E}}[\varphi(\bar{B}_{t_{i+1}^{n}})]=\mathbb{\tilde{E}}%
[\varphi(\bar{B}_{t_{i}^{n}}+B_{t_{i+1}^{n}}-B_{t_{i}^{n}}-\xi_{i}(t_{i+1}%
^{n}-t_{i}^{n})-\eta_{i}(\langle B\rangle_{t_{i+1}^{n}}-\langle B\rangle
_{t_{i}^{n}}))].
\]
By (7) in Theorem \ref{the6.1}, we get%
\begin{align*}
&  \mathbb{\tilde{E}}[\varphi(\bar{B}_{t_{i+1}^{n}})]\\
&  =\mathbb{\tilde{E}}[\varphi(x+B_{t_{i+1}^{n}}-B_{t_{i}^{n}}-b(t_{i+1}%
^{n}-t_{i}^{n})-d(\langle B\rangle_{t_{i+1}^{n}}-\langle B\rangle_{t_{i}^{n}%
}))]_{x=\bar{B}_{t_{i}^{n}},b=\xi_{i},d=\eta_{i}}\\
&  =\mathbb{\tilde{E}}[\mathbb{\hat{E}}[\varphi(x+B_{t_{i+1}^{n}}-B_{t_{i}%
^{n}})]_{x=\bar{B}_{t_{i}^{n}}}].
\end{align*}
Repeat this process, we obtain $\mathbb{\tilde{E}}[\varphi(\bar{B}%
_{t_{i+1}^{n}})]=\mathbb{\hat{E}}[\varphi(B_{t_{i+1}^{n}})]$. Similarly, we
can get%
\[
\mathbb{\tilde{E}}[\Phi(\bar{B}_{t_{1}},\bar{B}_{t_{2}}-\bar{B}_{t_{1}}%
,\ldots,\bar{B}_{t_{n}}-\bar{B}_{t_{n-1}})]=\mathbb{\hat{E}}[\Phi(B_{t_{1}%
},B_{t_{2}}-B_{t_{1}},\ldots,B_{t_{n}}-B_{t_{n-1}})].
\]

Step 3. For general bounded processes $(b_{t})$ and $(d_{t})$, we can choose
uniformly bounded processes $(b_{t}^{n})$, $(d_{t}^{n})\in M_{G}^{2,0}(0,T)$
such that $||b^{n}-b||_{M_{G}^{2}}+||d^{n}-d||_{M_{G}^{2}}\rightarrow0$. By
Proposition \ref{pro3.5}, we obtain the result by letting $n\rightarrow\infty$.
\end{proof}

\begin{remark}
\label{rem6.3} If $b_{s}=0$, we know by Remark \ref{rem5.3}%
\[
\mathbb{\tilde{E}}_{t}[\xi]=\mathbb{\hat{E}}_{t}[\xi\exp(\int_{t}^{T}%
d_{s}dB_{s}-\frac{1}{2}\int_{t}^{T}|d_{s}|^{2}d\langle B\rangle_{s})].
\]
This type of Girsanov transformation was studied in \cite{XSZ, Os}, but here
we give a simple proof. If $b_{s}\not =0$, we know by Theorem \ref{the5.2}%
\begin{align*}
\mathbb{\tilde{E}}_{t}[\xi]=  &  \mathbb{\hat{E}}_{t}^{\tilde{G}}[\xi\exp
(\int_{t}^{T}d_{s}dB_{s}-\frac{1}{2}\int_{t}^{T}|d_{s}|^{2}d\langle
B\rangle_{s}-\int_{t}^{T}b_{s}d_{s}ds\\
&  +\int_{t}^{T}b_{s}d\tilde{B}_{s}-\frac{1}{2}\int_{t}^{T}|b_{s}|^{2}%
d\langle\tilde{B}\rangle_{s})],
\end{align*}
where $(B,\tilde{B})$ is an auxiliary extended $\tilde{G}$-Brownian motion
and
\[
\tilde{G}(A)=\frac{1}{2}\sup_{\underline{\sigma}^{2}\leq v\leq\bar{\sigma}%
^{2}}\mathrm{tr}\left[  A\left[
\begin{array}
[c]{cc}%
v & 1\\
1 & v^{-1}%
\end{array}
\right]  \right]  ,\ A\in\mathbb{S}_{2}.
\]

\end{remark}

We now consider the Girsanov transformation for the case $d>1$. Let
$B_{t}=(B_{t}^{i})_{i=1}^{d}$ be a $d$-dimensional $G$-Brownian motion. We
consider the following $G$-BSDE:%
\[
Y_{t}=\xi+\int_{t}^{T}b_{s}Z_{s}ds+\int_{t}^{T}d_{s}^{ij}Z_{s}d\langle
B^{i},B^{j}\rangle_{s}-\int_{t}^{T}Z_{s}dB_{s}-(K_{T}-K_{t}),
\]
where $(b_{t})_{t\leq T}$ and $(d_{t}^{ij})_{t\leq T}$ are $\mathbb{R}^{d}%
$-valued bounded processes. By Theorem \ref{the6.1}, $\mathbb{\tilde{E}}%
_{t}[\xi]:=Y_{t}$ is a consistent sublinear expectation.

\begin{theorem}
\label{the6.4} (Girsanov Theorem) Let $(b_{t})_{t\leq T}$ and $(d_{t}%
^{ij})_{t\leq T}$ be $\mathbb{R}^{d}$-valued bounded processes. Then $\bar
{B}_{t}:=B_{t}-\int_{0}^{t}b_{s}ds-\int_{0}^{t}d_{s}^{ij}d\langle B^{i}%
,B^{j}\rangle_{s}$ is a $d$-dimensional $G$-Brownian motion under
$\mathbb{\tilde{E}}$.
\end{theorem}

\begin{proof}
The proof is similar to Theorem \ref{the6.2}.
\end{proof}


\renewcommand{\refname}{\large References}{\normalsize \ }

\end{document}